\documentclass[12pt]{amsart}
\usepackage{url, mathrsfs}
\usepackage{amssymb}
\usepackage{hyperref}
\usepackage{amsmath}
\usepackage{graphicx, 
epstopdf}
\usepackage{subfig}
\usepackage{color}
\usepackage{bbm}
\usepackage{subfloat}
\usepackage[draft]{fixme}
\usepackage{bm} 
\usepackage{bbm} 
\usepackage{epigraph}
\usepackage{hyperref}
\usepackage{endnotes}
\usepackage{ifpdf}
\usepackage{array}
\usepackage{wrapfig}
\usepackage{longtable}

\usepackage{multicol}
\usepackage{multirow}
\usepackage{graphicx}
\usepackage{tikz}
\usepackage{tikz-cd}
\usepackage{makecell}

\newcommand{\N}{\mathbb{N}}
\newcommand{\Z}{\mathbb{Z}}
\newcommand{\R}{\mathbb{R}}
\newcommand{\F}{\mathbb{F}}
\newcommand{\C}{\mathbb{C}}

\renewcommand{\P}{\mathbb{P}}

\newcommand{\x}{\mathbf{x}}

\theoremstyle{plain}
\newtheorem{thm}{Theorem}[section]
\newtheorem{prop}[thm]{Proposition}
\newtheorem{cor}[thm]{Corollary}
\newtheorem{lem}[thm]{Lemma}
\newtheorem{question}[thm]{Question}

\newtheorem*{claim*}{Claim}
\newtheorem*{thm*}{Theorem}

\newtheoremstyle{case}{}{}{}{}{}{:}{ }{}
\theoremstyle{case}

\theoremstyle{definition}

\newtheorem{example}[thm]{Example}
\newtheorem{remark}[thm]{Remark}

\DeclareMathOperator{\Gr}{Gr}

\DeclareMathOperator{\SL}{SL}
\DeclareMathOperator{\diag}{diag}

 \def\dashmapsto{\mapstochar\dashrightarrow}

\usepackage[inner=1in,outer=1in,top=1in,bottom=1in]{geometry}

\begin{document}
\nocite{*}
\title[Characterizing principal minors via determinantal polynomials]{
Characterizing principal minors of symmetric matrices via determinantal multiaffine polynomials
}

\author{Abeer Al Ahmadieh}
\address{Department of Mathematics, University of Washington, Seattle, WA, USA 98195} 
\email{aka2222@uw.edu}

 \author{Cynthia Vinzant}
\address{Department of Mathematics, North Carolina State University, Raleigh, NC, USA 27695}
\email{clvinzan@ncsu.edu}

\begin{abstract}
Here we consider the image of the principal minor map 
of symmetric matrices over an arbitrary unique factorization domain $R$. 
By exploiting a connection with symmetric determinantal representations, 
we characterize the image of the principal minor map 
through the condition that certain polynomials coming from so-called 
Rayleigh differences are squares in the polynomial ring over $R$.
In almost all cases, one can characterize the image of the principal minor map using the orbit of Cayley's hyperdeterminant under 
the action of $(\SL_2(R))^{n} \rtimes S_{n}$.
Over $\C$, this recovers a characterization of Oeding from 2011, and over $\R$, 
the orbit of a single additional quadratic inequality suffices to cut out the image.
Applications to other symmetric determinantal representations are also discussed.
\end{abstract}

\maketitle

\section{Introduction}
Given an $n\times n$ matrix $A$ with entries in a commutative ring $R$, let $A_S$ denote the principal minor obtained by taking the 
determinant of the principal submatrix of $A$ with rows and columns indexed by $S$. 
Let ${\rm Sym}_n(R)$ denote the space of symmetric $n\times n$ matrices with entries in $R$.  
The {\bf principal minor map} is the map 
\[
\varphi: {\rm Sym}_n(R) \to R^{2^n} \ \ \text{ given by } \ \ \varphi(A) = (A_S)_{S\subseteq [n]},
\]
where we take $A_{\emptyset} = 1$. 
Here we seek to characterize the image of the principal minor map over arbitrary unique factorization domains $R$, and in particular, arbitrary  fields.  

Over $\R$ and $\C$, this problem was studied by Holtz and Sturmfels \cite{HoltzSturmfels07}, who show that the image is invariant under an action of 
$\SL_2(\mathbb{R})^n \rtimes S_n$ and conjectured that the vanishing of 
polynomials in the orbit of the \emph{hyperdeterminant} under this group cuts out  
the image of the principal minor map over $\C$.  
This conjecture was resolved by Oeding in 2011, 
using tools from representation theory \cite{Oeding11}. 
Here we generalize this result to hold over arbitrary unique factorization domains, except those with exactly three elements.

We will study this problem by associating to the matrix $A$ the multiaffine polynomial
\[
f_A = \det\left({\rm diag}(x_1, \hdots, x_n) + A \right) = \sum_{S\subseteq [n]} A_S \prod_{i\not\in S}x_i. 
\]
This translates the problem of characterizing the image of the principal minor map 
to the problem of characterizing multiaffine polynomials in $R[x_1, \hdots, x_n]$
with symmetric determinantal representations.  Key to this characterization will 
be Rayleigh differences. 

The Rayleigh difference of a polynomial $f$ with respect to $i,j\in [n]$ is defined to be 
\begin{equation}\label{eq:DeltaDef}
\Delta_{ij}(f)  = \frac{\partial f}{\partial x_i}  \frac{\partial f}{\partial x_j} - f \frac{\partial^2 f}{\partial x_i \partial x_j}.
\end{equation}
These polynomials play a prominent role in the theory of stable polynomials \cite{Branden07}.  Using Dodgson condensation \cite{Dodgson}, one can see that for the determinantal polynomial $f_A$, all Rayleigh differences 
$\Delta_{ij}(f_A)$ are squares in the polynomial ring $R[x_1, \hdots, x_n]$. 
In 2015, Kummer, Plaumann and the second author prove the converse over $\R$ \cite{KPV15} 
and here we prove it over an arbitrary unique factorization domain. 

Formally, to ${\bf a} = (a_S)_{S\subseteq [n]}$ in $R^{2^n}$ we associate the polynomial 
$f_{\bf a} = \sum_{S\subseteq [n]}a_S \prod_{i\not\in S}x_i.$
\newtheorem*{thm:DeltaMinors}{Theorem~\ref{thm:DeltaMinors}}
\begin{thm:DeltaMinors}
Let $R$ be a unique factorization domain. 
An element ${\bf a} = (a_S)_{S\subseteq [n]}$ in $R^{2^n}$ is in the image of ${\rm Sym}_n(R)$ under
the principal minor map if and only if $a_{\emptyset}=1$ and for every $i,j\in [n]$,
$\Delta_{ij}(f_{\bf a})$ is a square in $R[x_1, \hdots, x_n]$. 
\end{thm:DeltaMinors}

For $n=3$, $\Delta_{12}(f_{\bf a})$ is a quadratic polynomial in the remaining variable $x_3$, namely
\[
\Delta_{12}(f_{\bf a}) = (a_1a_2 - a_{\emptyset} a_{12})x_3^2 + ( a_{1} a_{23} +a_{2} a_{13}-a_{3} a_{12} - a_{\emptyset} a_{123})x_3 + (a_{13}a_{23} - a_3a_{123}).
\]
For this polynomial to be a square, its discriminant must vanish, giving us a necessary 
equation on the coefficients of $f_{\bf a}$. 
The discriminant of $\Delta_{12}(f_{\bf a})$ with respect to $x_3$ equals the well-known Cayley $2\times 2 \times 2$ hyperdeterminant, 
\begin{align*}
{\rm HypDet}({\bf a}) 
=&  \ ( a_{1} a_{23} +a_{2} a_{13}-a_{3} a_{12} - a_{\emptyset} a_{123})^2 -  4(a_1a_2 - a_{\emptyset} a_{12}) (a_{13}a_{23} - a_3a_{123}) \\ 
=& \ a_{\emptyset}^2 a_{123}^2 +  a_{1}^2 a_{23}^2 +a_{2}^2 a_{13}^2+a_{3}^2 a_{12}^2  - 2 a_{\emptyset} a_{1} a_{23} a_{123} - 2 a_{\emptyset} a_{2} a_{13} a_{123} -2 a_{\emptyset} a_{3} a_{12} a_{123}\\
 &-2 a_{1} a_{2} a_{13} a_{23} - 2 a_{1} a_{3} a_{12} a_{23} -2 a_{2} a_{3} a_{12} a_{13} + 4 a_{\emptyset} a_{23} a_{13} a_{12}+ 4 a_{123} a_{1} a_{2} a_{3}.
 \end{align*}
The coefficients of $1$ and $x_3^2$ in $\Delta_{12}(f_{\bf a})$ are $a_{13}a_{23} - a_3a_{123}$ and $a_1a_2 - a_{\emptyset} a_{12}$, respectively. We see that 
$\Delta_{12}(f_{\bf a})$ is a square if and only if these two coefficients are squares in $R$ 
and the discriminant, ${\rm HypDet}({\bf a}) $, is zero. 
One can check that ${\rm Discr}_{x_3}\Delta_{12}(f_{\bf a})$, 
${\rm Discr}_{x_2}\Delta_{13}(f_{\bf a})$ and ${\rm Discr}_{x_1}\Delta_{23}(f_{\bf a})$ 
are all the same and equal to ${\rm HypDet}({\bf a})$. 
Therefore a vector ${\bf a}\in R^{2^3}$ with $a_{\emptyset}=1$ is in the image of the principal minor map if and only if ${\rm HypDet}({\bf a})=0$ and for every $i,j\in [3]$ with 
$\{k\} = [3]\backslash \{i,j\}$, 
both $a_{ik}a_{jk} - a_ka_{ijk}$ and $a_ia_j - a_{\emptyset} a_{ij}$ are squares in $R$. 

Our main result is that, under the action of $\SL_2(R)^n \rtimes S_n$, these conditions 
characterize the image of the principal minor map for general $n$. 

\newtheorem*{thm:HypDet}{Theorem~\ref{thm:HypDet2}}
\begin{thm:HypDet} Let $R$ be a unique factorization domain with $|R|\neq 3$ and 
${\bf a} = (a_S)_{S\subseteq [n]}\in R^{2^n}$ with $a_{\emptyset}=1$. There exists a symmetric matrix over $R$ with principal minors ${\bf a}$ if and only if 
\begin{itemize}
\item[(i)] for every $i,j\in [n]$, $a_{i}a_{j} - a_{ij} $ is a square in $R$, and  
\item[(ii)] for every $\gamma\in \SL_2(R)^n \rtimes S_n$, $(\gamma\cdot{\rm HypDet})({\bf a}) =0.$
\end{itemize}
\end{thm:HypDet}

While the description in (ii) involves a potentially infinite set of quartic polynomials, 
we give an explicit set of $\binom{n}{3}5^{n-3}$ elements  $\gamma\in \SL_2(R)^n \rtimes S_n$
that are necessary and sufficient in this characterization (see Remark~\ref{rem:count}). 
As observed in \cite[Observation III.15]{OedingThesis}, when $R$ is a field of characteristic zero, 
this is precisely the dimension of the linear space 
in $R[a_S : S\subseteq [n]]$ spanned by the polynomials $(\gamma\cdot{\rm HypDet})({\bf a})$.

As a corollary of Theorem~\ref{thm:HypDet2}, we obtain another proof of Oeding's result over $\C$: 

\newtheorem*{cor:HypDetC}{Corollary~\ref{cor:HypDetC}}
\begin{cor:HypDetC}
Let ${\bf a} = (a_S)_{S\subseteq [n]}\in \C^{2^n}$ with $a_{\emptyset}=1$. 
Then ${\bf a}$ belongs to the image of the principal minor map over $\C$ if and only if 
for every $\gamma\in \SL_2(\mathbb{C})^n \rtimes S_n$, $(\gamma\cdot{\rm HypDet})({\bf a}) =0.$ 
\end{cor:HypDetC}

We also get a semialgebraic description of the image of the principal minor map over $\R$.

\newtheorem*{cor:HypDetR}{Corollary~\ref{cor:HypDetR}}
\begin{cor:HypDetR}
Let ${\bf a} = (a_S)_{S\subseteq [n]}\in \R^{2^n}$ with $a_{\emptyset}=1$. 
Then ${\bf a}$ belongs to the image of the principal minor map over $\R$ if and only if 
for every $i,j\in [n]$, 
$a_{ i}a_{ j} - a_{ ij} \geq 0$ and 
for every $\gamma\in \SL_2(\mathbb{R})^n \rtimes S_n$, 
$(\gamma\cdot{\rm HypDet})({\bf a}) =0$.
\end{cor:HypDetR}

For real symmetric matrices, the inequalities $A_iA_j - A_{ij}\geq 0$ are 
a subset of the well-known Hadamard-Fischer inequalities 
$A_{S\cup i}A_{S\cup j} - A_SA_{S\cup ij} \geq 0$, which were also used by Holtz and Sturmfels in 
a partial characterization of the image of the principal minor map over $\R$, \cite[Theorem~6]{HoltzSturmfels07}.
Corollary~\ref{cor:HypDetR} states that these inequalities and the equations given by the 
images of the $2\times 2\times 2$ hyperdeterminant under $\SL_2(\mathbb{R})^n \rtimes S_n$ cut out the image of the principal minor map over $\R$. 
The image of the principal minor map over $\R$ is of special interest, 
as for positive semidefinite matrices $A$, the discrete probability measure on $2^{[n]}$ 
given by ${\rm Prob}(S) \propto A_S$ forms a determinantal point process. 
These distributions have several nice properties, such as negative association, and 
appear in a wide range of applications \cite{BBL09, DPP&MachineLearning}. 

Over fields of characteristic two, the discriminant of a univariate quadratic is a square. 
From Theorem~\ref{thm:HypDet2}, we then recover the results of van Geeman and Marrani \cite{char2} 
that the image is cut out by quadratic equations: 

\newtheorem*{cor:HypDetF2}{Corollary~\ref{cor:HypDetF2}}
\begin{cor:HypDetF2}
Let ${\bf a} = (a_S)_{S\subseteq [n]}\in R^{2^n}$ with $a_{\emptyset}=1$ where $R$ has characteristic two. 
 There exists a symmetric matrix over $R$ with principal minors ${\bf a} $
if and only if 
\begin{itemize}
\item[(i)] for every $i,j\in [n]$, $a_{i}a_{j} - a_{ij} $ is a square in $R$, and  
\item[(ii)] for every $\gamma\in \SL_2(\F_2)^n \rtimes S_n$, $\gamma\cdot(a_{\emptyset} a_{123} +  a_{1} a_{23} +a_{2} a_{13}+a_{3} a_{12})=0$.
\end{itemize}
In particular, for $R = \F_2$, {\rm (i)} is always satisfied and the image of the principal minor map is cut out by the quadratic equations in {\rm (ii)}. 
\end{cor:HypDetF2}

Several other incarnations of the principal minor map have been studied. 
Lin and Sturmfels \cite{LinSturmfels09} prove that the ideal of the image of the space of general $4\times 4$ complex matrices under the principal minor map is minimally generated by $65$ polynomials of degree $12$ and they conjecture that the image of the space of general square complex matrices is cut out by equations of degree $12$. Huang and Oeding \cite{HuangOeding17} solve this conjecture in the special case where all principal minors of same size are equal (\textit{the symmetrized principal minor assignment problem}). They provide a minimal parametrization of the respective varieties in the cases of symmetric, skew symmetric and square complex matrices. Kenyon and Pemantle \cite{KenyonPemantle14} adjust the principal minor map by adding \textit{the almost principal minors} to the vector in the image and they showed that the ideal of the variety in this case is generated by translations of a single relation. 
In future work, we intend to extend the techniques in this paper to other spaces of (non-symmetric) matrices. 

Griffin and Tsatsomeros \cite{GriffinTsatsomeros06,GriffinTsatsomeros06s} 
examine the complexity of computing the vector of the principal minors of an $n\times n$ matrix 
and give a numerical algorithm that reconstructs a preimage matrix, if it exists over $\C$, from such a vector. 
Rising, Kulesza and Taskarc \cite{RisingKuleszaTaskar15} provide an efficient algorithm for reconstruction in the symmetric case.

The paper is organized as follows. In Section~\ref{sec:background}, we establish 
notation and introduce the action of $\SL_2(R)^n \rtimes S_n$. 
In Section~\ref{sec:DetRep}, we establish the connection between square Rayleigh differences and determinantal representations and prove Theorem~\ref{thm:DeltaMinors}. 
In Section~\ref{sec:multiQsquares}, we use the group action of $\SL_2(R)^n \rtimes S_n$
to characterize the set multiquadratic squares in $R[x_1, \hdots, x_n]$ and 
use this to characterize the image of the principal minor map in Section~\ref{sec:finale}. 
To conclude, in Section~\ref{sec:otherDetRep}, we discuss some consequences for other determinantal representations as well as connections to the Grassmannian $\Gr_{\F}(d,n)$ over arbitrary fields.

{\bf Acknowledgements}. 
Both authors were partially supported by the National Science Foundation Grant No.~DMS-1620014 and DMS-1943363.  This material is based upon work directly supported by the National Science Foundation Grant No.~DMS-1926686, and indirectly supported by the National Science Foundation Grant No.~CCF-1900460.

\section{Background and notation}\label{sec:background}
Throughout the paper, we take $R$ to be a unique factorization domain. 
Let $R[{\bf x}]$ denote the polynomial ring $R[x_1, \hdots, x_n]$. 
For $\alpha = (\alpha_1, \hdots, \alpha_n) \in \N^n$ and $S\subseteq [n]$, we use the notation ${\bf x}^{\alpha}$ for $\prod_{i=1}^nx_i^{\alpha_i}$ and ${\bf x}^S$ for $\prod_{i\in S}x_i$. 
For $f\in R[{\bf x}]$, let $\deg_i(f)$ denote the degree of $f$ in the variable $x_i$. 
Given ${\bf d} = ({\bf d}_1, \hdots, {\bf d}_n) \in \Z_{\geq 0}^n$, let $R[{\bf x}]_{\leq {\bf d}}$ denote the set of 
polynomials with degree at most ${\bf d}_i$ in $x_i$ for each $i=1, \hdots, n$. 
These form an $R$-module of rank $\prod_{i=1}^n ({\bf d}_i+1)$. 
When ${\bf d}_1 = \hdots = {\bf d}_n = m$, we abbreviate  $R[{\bf x}]_{\leq (m,\hdots m)}$ by $R[{\bf x}]_{\leq {\bf m}}$.  
Of particular interest are \emph{multiaffine polynomials}, $R[{\bf x}]_{\leq {\bf 1}}$, with degree $\leq 1$ in each variable, and \emph{multiquadratic polynomials}, $R[{\bf x}]_{\leq {\bf 2}} $, with degree $\leq 2$ in each variable. 
We will often consider multi-homogenezations of these polynomials. 
Let $R[{\bf x}, {\bf y}]_{\bf d}$ denote the set of polynomials in the variables $x_1, \hdots, x_n$ and $y_1, \hdots, y_n$ 
that are homogeneous of degree ${\bf d}_i$ in each pair of variables $x_i, y_i$. 
For $f = \sum_{\alpha}c_{\alpha}{\bf x}^{\alpha}$, 
let 
$f^{\rm {\bf d}-hom}$ in $ R[{\bf x}, {\bf y}]_{\bf d}$ denote the polynomial 
\[
f^{\rm {\bf d}-hom} = \prod_{i=1}^n y_i^{{\bf d}_i} \cdot f\left(x_1/y_1, \hdots, x_n/y_n\right)
= \sum_{\alpha}c_{\alpha}{\bf x}^{\alpha}{\bf y}^{{\bf d} - \alpha}. 
\]

To a polynomial $f \in R[{\bf x}]_{\leq {\bf 2}}$, its discriminant with respect to any variable $x_k$, denoted ${\rm Discr}_{x_k}(f)$, 
equals $b^2 - 4ac$ where $f = ax_k^2 + bx_k + c$ and  $a$, $b$, $c$ do not involve the variable $x_k$. 
Similarly, for a multiquadratic polynomial $f \in R[{\bf x},{\bf y}]_{{\bf 2}}$, we can write 
$f = ax_k^2 + bx_ky_k + cy_k^2$ and define its discriminant with respect to $(x_k, y_k)$ 
to be ${\rm Discr}_{(x_k,y_k)}(f) = b^2 - 4ac$.

The symmetric group acts on $R[{\bf x}]$ by permuting the variables. That is, for $\pi\in S_n$, $\pi\cdot f$ equals $f(x_{\pi(1)}, \hdots, x_{\pi(n)})$. 
The action of $\SL_2(R)^n$ on $R[{\bf x}]_{\leq {\bf d}}$ is defined as follows. 
Let $\gamma =(\gamma_i)_{i\in [n]} $ in $\SL_2(R)^{n}$ where $\gamma_i = \tiny{\begin{pmatrix} a_i& b_i  \\ c_i & d_i \end{pmatrix}}$. Then for $f \in R[{\bf x}]_{\leq {\bf d}}$, 
\[
\gamma\cdot f = \prod_{i=1}^n (c_i x_i + d_i)^{{\bf d}_i} \cdot f\left(\frac{a_1 x_1 + b_1}{c_1 x_1 + d_1}, \hdots, \frac{a_n x_n + b_n}{c_n x_n + d_n}\right).
\]
One way to interpret this action is via the multi-homogenezation of $f$. 
The induced action of $\gamma$ on $f^{\rm {\bf d}-hom} $ is just an $R$-linear change of coordinates: 
\[\gamma \cdot f^{\rm {\bf d}-hom} = f^{\rm {\bf d}-hom}\left(\gamma_1\cdot\begin{pmatrix} x_1 \\ y_1\end{pmatrix}, \hdots,  \gamma_n\cdot\begin{pmatrix} x_n \\ y_n\end{pmatrix}\right).\]
Restricting to $y_1 = \hdots = y_n = 1$ gives back $\gamma\cdot f $.
Similarly, we can extend the action of $S_n$ to $R[{\bf x}, {\bf y}]_{\bf d}$
by simultaneous permutations of the $x_i$ and $y_i$ coordinates, i.e. $\pi \cdot f = f(x_{\pi(1)}, \hdots, x_{\pi(n)}, y_{\pi(1)}, \hdots, y_{\pi(n)})$. 

Note that  $R[{\bf x}]_{\leq {\bf 1}}$  and $R^{2^n}$ are isomorphic $R$-modules, 
and so the action of $\SL_2(R)^n\rtimes S_n$ on  $R[{\bf x}]_{\leq {\bf 1}}$ also gives one on $R^{2^n}$. 
Specifically, to an element ${\bf a} = (a_S)_{S\subseteq [n]}$ in $R^{2^n}$ we associate the multiaffine polynomial 
$f_{\bf a} = \sum_{S\subseteq [n]}a_S {\bf x}^{[n]\backslash S}$
and to any polynomial $f\in R[{\bf x}]_{\leq {\bf 1}}$ we associate the point ${\bf a} = (a_S)_{S\subseteq [n]}$ in $R^{2^n}$
with $a_S = {\rm coeff}(f, {\bf x}^{[n]\backslash S})$. 
Note that if $A$ is a symmetric matrix with $a_S = A_S$, then 
$f_{\bf a} = \det\left({\rm diag}(x_1, \hdots, x_n) + A\right)$. 
For any  $\gamma\in \SL_2(R)^n\rtimes S_n$, we define $\gamma\cdot {\bf a}$ by the relation
$f_{\gamma \cdot {\bf a}} = \gamma \cdot f_{\bf a}$. 
Similarly, we define the action of $\SL_2(R)^n\rtimes S_n$ on the polynomial ring 
$R[a_S : S\subseteq [n]]$ by $\gamma\cdot F({\bf a})  = F(\gamma\cdot{\bf a})$. 

\begin{example} For $n=3$, consider 
$\gamma = \left(\tiny{\begin{pmatrix}0 & -1 \\ 1 & 0  \end{pmatrix}}, {\rm Id}_2, {\rm Id}_2\right)$ in $\SL_2(R)^3$. For any point ${\bf a} = (a_S)_{S\subseteq[3]} \in R^{2^3}$, 
\[\gamma \cdot f_{\bf a} = x_1 f_{\bf a}(-x_1^{-1}, x_2, x_3)
=  \sum_{T\ni 1}  a_T{\bf x}^{[3]\backslash (T\backslash 1)} -  \sum_{T\not\ni 1} a_T{\bf x}^{[3]\backslash (T\cup 1)}
=  \sum_{S\not\ni 1}  a_{S\cup 1}{\bf x}^{[3]\backslash S} -  \sum_{S\ni 1} a_{S\backslash 1}{\bf x}^{[3]\backslash S}. 
\]
Taking coefficients of $\gamma \cdot f_{\bf a}$ shows that 
 $(\gamma\cdot {\bf a})_S$ equals $a_{S\cup 1}$ if $1\not\in  S$ and 
$-a_{S\backslash 1}$ if $1\in S$.  For $F({\bf a}) = a_2a_3 - a_{\emptyset}a_{23}$, 
we see that $\gamma\cdot F({\bf a}) = F(\gamma \cdot {\bf a}) = 
a_{12}a_{13} - a_1 a_{123}$.  From this we see that 
the image of $F$ under the group  $\SL_2(R)^3\rtimes S_3$ includes 
all six polynomials of the form $a_{i}a_j - a_{\emptyset}a_{ij}$ and $a_{ik}a_{jk} - a_{k}a_{ijk}$ 
for $\{i,j,k\} = \{1,2,3\}$.  
\end{example}

\begin{prop}\label{prop:SL2onDelta}
Consider an element $\gamma \in \SL_2(R)^{n}$ that acts by 
$\tiny{\begin{pmatrix} a&  b \\ c & d \end{pmatrix}}$ in the $k$-th coordinate 
and the identity in all others.  For any $f\in R[{\bf x}]_{\leq {\bf 1}}$, 
\[
\Delta_{ij}(\gamma \cdot f) = \begin{cases} 
 \Delta_{ij}(f) & \text{ if } k = i, j \\
\gamma \cdot \Delta_{ij}(f) & \text{ otherwise.}
\end{cases}
 \]
 \end{prop}
 \begin{proof}
For each $k\in [n]$, let $f_{k} $ denote the derivative of $f$ with respect to $x_k$ and let 
 $f^{k}$ denote its specialization to $x_k=0$.  We can then write $f = x_{k}f_{k} + f^{k}$. 
 Then 
 \[
 \gamma\cdot f  \  =  \  
 (a x_k + b)f_k + (cx_k + d) f^k
\  =  \ x_k (af_k + cf^k) + (bf_k + df^k).
 \]
In particular, $\frac{\partial}{\partial x_{k}}( \gamma\cdot f ) = (af_k + cf^k)$ and 
$( \gamma\cdot f )|_{x_{k}=0} = bf_k + df^k$.

To see how this action affects Rayleigh differences, 
we write the polynomial $\Delta_{ij}(f)$ as 
\[
\Delta_{ij}( f) = f_if_j - f f_{ij} = f_{i}^jf_j^i - f^{ij} f_{ij}, 
\]
where $f_i^j$ for example denotes $\frac{\partial f}{\partial x_i}|_{x_j = 0}$. 
If $k=i$, applying $\gamma$ then gives 
\[
\Delta_{ij}(\gamma\cdot f)  = (a f_{i}^j +  c f^{ij}) (bf_{ij}+ df_j^i) - (bf_i^{j}+df^{ij}) (af_{ij} + c f^i_{j}) = (ad-bc) (f_{i}^jf_j^i - f^{ij} f_{ij} ),\]
showing that $\Delta_{ij}(f)$ is invariant. Another way to see this is to view $\Delta_{ij}(f)$ as the resultant 
of $f_j$ and $f^j$ with respect to $x_i$.  
Note that for $k\neq i,j$, $\gamma$ commutes with taking the derivatives with respect to $x_i$, $x_j$ 
and restricting $x_i$, $x_j$ to zero. 
Therefore $\Delta_{ij}(\gamma \cdot f) = \gamma \cdot \Delta_{ij}(f)$. 
 \end{proof}

\begin{cor}\label{cor:SquareDelInvariance}
The set of polynomials $f\in R[{\bf x}]_{\leq 1}$ such that $\Delta_{ij}(f)$ 
is a square for all $i,j\in [n]$ is invariant under the action of $\SL_2(R)^n \rtimes S_n$.
\end{cor}
\begin{proof}
Note that the set of squares in $R[{\bf x}]_{\leq {\bf 2}}$ is invariant 
under the action of $\SL_2(R)^n \rtimes S_n$. 
If $g = h^2$ where $h\in R[{\bf x}]_{\leq {\bf 1}}$
then for $\pi\in S_n$, $\pi\cdot g = (\pi \cdot h)^2$. 
Similarly for $\gamma \in \SL_2(R)^n $, $\gamma \cdot g = (\gamma \cdot h)^2$. 
Note here that $\gamma$ acts on $g$ as an element of 
$R[{\bf x}]_{\leq {\bf 2}}$ and acts on $h$ as an element of 
$R[{\bf x}]_{\leq {\bf 1}}$, regardless of their degrees. 

First we note invariance under the symmetric group. For any $\pi\in S_n$, 
$\Delta_{ij}(\pi \cdot f)$ equals $\pi\cdot \Delta_{\pi^{-1}(i)\pi^{-1}(j)}(f)$. 
Therefore if $f$ has the property that 
$\Delta_{ij}(f)$ is a square for all $i,j$, then so does $\pi\cdot f$. 
Similarly, by Proposition~\ref{prop:SL2onDelta}, 
for any $\gamma\in \SL_2(R)^n$, if $\Delta_{ij}(f)\in R[{\bf x}]_{\leq {\bf 2}}$ is a square, then so is  $\Delta_{ij}(\gamma\cdot f)$. 
\end{proof}

We will also use the usual homogenization of a polynomial to some total degree $d$, using a single homogenizing variable $y$. 
That is, for $f = \sum_{\alpha}c_{\alpha}{\bf x}^{\alpha}\in R[{\bf x}]$ of total degree $d=\deg(f)$, 
its homogenization is 
\[
\overline{f} = y^d f\left(x_1/y, \hdots, x_n/y\right)=
 \sum_{\alpha}c_{\alpha}{\bf x}^{\alpha}y^{d - |\alpha|}  \ \in \ R[{\bf x}, y].
 \]
To end this section, we remark that the condition that $\Delta_{ij}(f)$ is 
a square is robust to various homogenizations. 

\begin{prop}\label{prop:homDelta}
	Let $f\in R[{\bf x}]_{\leq {\bf 1}}$ and let $\overline{f}$ denote the homogenization of 
	$f$ in $R[{\bf x},y]$. Then the following are equivalent
	\begin{center}\begin{tabular}{ll}
	{\rm (a)} $\Delta_{ij}(f)$ is a square in $R[{\bf x}]$, 
	& {\rm (b)} $ \Delta_{ij}(\overline{f})$ is a square in $R[{\bf x},y]$,\smallskip \\
	{\rm(c)} $\overline{\Delta_{ij}(f)}$ is a square in $R[{\bf x},y]$, 
	& {\rm(d)} $(\Delta_{ij}(f))^{{\bf 2}-{\rm hom}}$ is a square in $R[{\bf x},{\bf y}]_{\bf 2}$.	
	\end{tabular}\end{center}
\end{prop}
\begin{proof}
The implications (b)$\Rightarrow$(a), (c)$\Rightarrow$(a), and (d)$\Rightarrow$(a)
follow from restricting to $y=1$ or $y_1 = \hdots = y_n = 1$. 
For (a)$\Rightarrow$(c) and (a)$\Rightarrow$(d), we note that if $\Delta_{ij}(f) = g^2$ 
for some $g\in R[{\bf x}]$, then $\overline{\Delta_{ij}(f)} = (\overline{g})^2$ and 
$(\Delta_{ij}(f))^{{\bf 2}-{\rm hom}} = (g^{{\bf 1}-{\rm hom}})^2$. 
For (a)$\Rightarrow$(b), let $f\in R[{\bf x}]_{\leq {\bf 1}}$ with total degree $d$ and suppose that 
$\Delta_{ij}f = g^2$ for some $g \in R[{\bf x}]$. Let $m = \deg(g)$. 
By definition, $\Delta_{ij}(\overline{f}) \in R[{\bf x}, y]$ 
is homogeneous of degree $2d-2$. Its restriction to $y=1$ equals $\Delta_{ij}f$. 
Therefore $\Delta_{ij}(\overline{f})$ equals $y^{2d-2-2m}\overline{\Delta_{ij}(f)}$, showing 
that  $\Delta_{ij}(\overline{f})$ can be written as $(y^{d-1-m}\overline{g})^2$. 
\end{proof}	

\section{Squares to Determinantal Representations}\label{sec:DetRep}

In this section we prove Theorem~\ref{thm:DeltaMinors}.
This relies heavily on the structure of the polynomials $\Delta_{ij}(f)$ 
defined in \eqref{eq:DeltaDef}. 
If $f$ is multiaffine, then $\Delta_{ij}(f)$ does not involve the variables $x_i$ and $x_j$ and 
has degree $\leq 2$ in the other variables. In particular, if $\Delta_{ij}(f) = (g_{ij})^2$ for some $g_{ij}\in R[x_1, \hdots, x_n]$, then $g_{ij}$ does not involve the variables $x_i$ and $x_j$ and 
has degree $\leq 1$ in the rest. 
We first work with more general determinantal representations in a larger ring $R[x_1, \hdots, x_m] = R[x_{n+1}, \hdots, x_m][{\bf x}]$ with $n<m$.

\begin{thm}\label{thm:DeltaDetRep}
Let $f \in R[x_1, \hdots, x_m]$ be a homogeneous polynomial of degree $n<m$. Suppose $f$ is multiaffine in the variables $x_1,\ldots,x_n$ and its coefficient of $x_1\cdots x_n$ equals one. 
Then $f = \det({\rm diag}(x_1,\hdots,x_n) + \sum_{j=n+1}^m x_j M_j)$ for some $M_j\in {\rm Sym}_n(R)$ if and only if  $\Delta_{ij} f$ is a square in $R[x_1\hdots, x_m]$ for all  $1 \leq i,j \leq n$.
\end{thm}
\begin{proof}[Proof of ($\Rightarrow$)]
This follows from a classical equality on the principal minors of an $n\times n$ matrix, used by 
Dodgson \cite{Dodgson} as a method for computing determinants.  For subsets $S, T\subset [n]$ of equal cardinality, 
let $M(S,T)$ denote the submatrix of $M$ obtained by \emph{dropping} rows $S$ and columns $T$ from $M$. 
Then for any $i\neq j\in [n]$, 
\begin{equation}\label{eq:dodgson} 
\det(M(i,i))\cdot \det(M(j,j))  - \det(M)\det(M(\{i,j\}, \{i,j\})) = \det(M(i,j))\cdot \det(M(j,i)).
\end{equation}
Note that for $M = {\rm diag}(x_1,\hdots,x_n) + \sum_{j=n+1}^m x_j M_j$ and any subset $S\subseteq [n]$, 
the principal minor $\det(M(S,S))$ equals the derivative of $f$ with respect to the variables in $S$,  
$\left(\prod_{i\in S}\frac{\partial}{\partial x_i}\right) f$. 
The equation above then gives that $\Delta_{ij}(f)$ equals $\det(M(i,j))\cdot \det(M(j,i))$. 
Since $M$ is symmetric, this shows that $\Delta_{ij}(f)= (\det(M(i,j)))^2$.
\end{proof}

We prove the other direction of this theorem after the following lemma. 

\begin{lem}\label{lem:factorDelta}
Let $f \in R[x_1, \hdots, x_m]$ be multiaffine in the variables $x_1,\ldots,x_n$ and its coefficient of $x_1\cdots x_n$ equal one.
If $f = g\cdot h$ for some $g,h\in R[x_1, \hdots, x_m]$, then $g$ and $h$ are multiaffine in 
disjoint subsets of the variables $x_1, \hdots, x_n$ and we can take their leading coefficients in these variables to be  one. 
Moreover, $\Delta_{ij} f$ is a square if and only if
 $\Delta_{ij} g$ and $\Delta_{ij} h$ are squares.
\end{lem}

\begin{proof}
For any $i\in [n]$, the degree of $f$ in $x_i$  must be the sum of the degrees of $g$ and $h$ in $x_i$. Since this sum of nonnegative numbers is one for each $i\in [n]$, 
we see that for some subset $I\subseteq[n]$, 
$g$ is multiaffine in $\{x_i : i\in I\}$, $h$ is multiaffine in $\{x_j : j\not\in I\}$,
and $\deg_i(h) = \deg_j(g) = 0$ for any $i\in I$ and $j\not\in I$. 

The highest degree term in $f$ with respect to the variables $x_1, \hdots, x_n$, 
$\prod_{i=1}^nx_i$, is the product of the highest degree terms in $g$ and $h$. 
Therefore for some $r, s\in R$, these terms are $r\prod_{i\in I}x_i$ and 
$s\prod_{j\not\in I}x_j$, respectively. Since $rs = 1$, we can replace $g$ with $sg$ and $h$ with 
$rh$ to obtain a factorization in which both have leading coefficient equal to $1$. 

For $i\in I$, $\partial(g\cdot h)/\partial x_i = h\cdot \partial g/\partial x_i$ 
and similarly, for $j\not\in I$, $\partial (g\cdot h)/\partial x_j = g\cdot \partial h/\partial x_j$. 
From this, one can check that 
$\Delta_{ij}(gh)$ equals  $h^2\Delta_{ij}(g)$ for $i,j \in I$, 
$g^2\Delta_{ij}(h) $ for $i,j\in [n]\backslash I$ and zero otherwise. 
In each case, we see that $\Delta_{ij}(gh)$ is a square in $R[x_1, \hdots, x_m]$ if and only if 
both $\Delta_{ij}(g)$ and $\Delta_{ij}(h)$ are squares. 
\end{proof}

\begin{proof}[Proof of Theorem~\ref{thm:DeltaDetRep}($\Leftarrow$)]
Let $S $ denote the ring $R[x_1, \hdots, x_m]$. 
Suppose that $f$ is irreducible in $S$. 
For each $i\in [n]$, let $g_{ii} = \frac{\partial f}{\partial x_i}$ 
and for each $i<j$, suppose that $\Delta_{ij}f$ equals $(g_{ij})^2$  for some $g_{ij}\in S$.
This implies that $\frac{\partial f}{\partial x_i}\cdot \frac{\partial f}{\partial x_j}$ is
equivalent to $(g_{ij})^2$ modulo $\langle f \rangle$. 
For $1<i<j$, the polynomials $(g_{11}g_{ij})^2$ and $(g_{1j}g_{i1})^2$ 
are both equivalent to $(\frac{\partial f}{\partial x_1})^2\frac{\partial f}{\partial x_i}\cdot \frac{\partial f}{\partial x_j}$, showing that 
\[(g_{11}g_{ij}- g_{1j}g_{i1})(g_{11}g_{ij}+ g_{1j}g_{i1}) =  (g_{11}g_{ij})^2 - (g_{1j}g_{i1})^2 \equiv 0 \ \ \mod \langle f \rangle. \]
Since $f$ is irreducible, $S/\langle f\rangle $ is an integral domain. 
Therefore one of the two factors above must be zero in $S/\langle f \rangle$. 
After changing the sign of $g_{ij}$ if necessary, we can assume that it is the first factor, 
giving that $g_{11}g_{ij}- g_{1j}g_{i1} \in \langle f \rangle$. 
Let $G\in {\rm Sym}_n(S)$ be the symmetric matrix with $(i,j)$th entry $g_{ij} = g_{ji}$. 
We claim that the $2\times 2$ minors of $G$ lie in $\langle f \rangle$.
Note that by construction, for any $i,j,k,l\in [n]$, 
	\[
	g_{11}^2 (g_{ij}g_{kl} - g_{il}g_{kj})  
	= 
	(g_{11} g_{ij})(g_{11} g_{kl}) - (g_{11}g_{il})(g_{11}g_{kj})
	\equiv 
	g_{1i}g_{1j}g_{1k}g_{1l} -  g_{1i}g_{1l}g_{1k}g_{1j}
	= 
	0	\!\! \mod \langle f \rangle.
	\]
Since $f$ is irreducible and $g_{11} = \partial f /\partial x_1$ has smaller degree, $g_{11}$ is not a zero-divisor in $S/\langle f\rangle$. Therefore the minor $g_{ij}g_{kl} - g_{il}g_{kj}$ belongs to $\langle f \rangle$. 

From this it follows that $f^{k-1}$ divides the $k\times k$ minors of $G$ 
for every $2\leq k\leq n$, see \cite[Lemma 4.7]{PV13}. 
In particular, $f^{n-2}$ divides the entries of the adjugate matrix $G^{\rm adj}$.  
Let 
\[
M = (1/f^{n-2})\cdot G^{\rm adj}.
\]  
Also $f^{n-1}$ divides $\det(G)$, and since these both have degree $n(n-1)$, 
there must be some constant $\lambda\in R$ for which $\det(G) = \lambda \cdot f^{n-1}$. 

We can see that $\lambda = 1$ by specializing $x_{\ell}$ to $0$ for all $\ell>n$. 
For any polynomial $h\in S$, let $h({\bf x}, 0)$ denote the 
specialization of $h$ with $x_{\ell} = 0$ for $\ell=n+1, \hdots, m$. 
Then $f({\bf x}, 0)$ equals $ x_1\cdots x_n$ and  $g_{ii}({\bf x}, 0) = \prod_{j\neq i} x_j$. 
Recall that $g_{ij} \in R[x_k: k\neq i,j]$ 
has total degree $n-1$ and degree at most one in each variable $x_k$ for $k\in [n]\backslash\{i,j\}$. Therefore every monomial appearing in $g_{ij}$ 
with non-zero coefficient must involve a variable $x_{\ell}$ for $\ell> n$, giving that $g_{ij}({\bf x}, 0) = 0$. 
Specializing all entries of $G$ to $x_{\ell} = 0$ for $\ell > n$, gives the diagonal matrix 
\[
G({\bf x},0) = {\rm diag}\left(\prod_{j\neq 1} x_j, \hdots, \prod_{j\neq n} x_j\right) = \prod_{j=1}^nx_j \cdot {\rm diag}\left(\frac{1}{x_1}, \hdots, \frac{1}{x_n}\right).
\]
Its determinant is $\prod_{i=1}^n x_i^{n-1}$ which equals $f({\bf x},0)^{n-1}$, showing that $\lambda=1$. 
From this and the equation $G\cdot G^{\rm adj} = \det(G)\cdot {\rm Id}_n$, it follows that 
	\[\det(M) \ \ = \ \ \frac{1}{f^{n(n-2)}}\cdot \det(G^{\rm adj})\ \ =\ \ 
	\frac{1}{f^{n(n-2)}} \det(G)^{n-1}\ \ =\ \   \frac{1}{f^{n(n-2)}} f^{(n-1)^2}\ \ =\ \  f.\]
Note that the entries of $M$ have degree $\leq (n-1)^2 - n(n-2) = 1$, so we can write 
$M$ as $\sum_{i=1}^m x_i M_i$ for some matrices $M_i\in {\rm Sym}_n(R)$. 
To finish the proof it suffices to show that $\sum_{i=1}^n x_i M_i = \diag(x_1, \hdots, x_n)$. 
Indeed, using the previous formula for $G({\bf x},0)$ we see that
	\begin{align*}
		M({\bf x},0) 	& = \frac{1}{\left(\prod_{j=1}^nx_j\right)^{n-2}}\cdot \left(\prod_{j=1}^nx_j \cdot {\rm diag}\left(\frac{1}{x_1}, \hdots, \frac{1}{x_n}\right)\right)^{\rm adj}\\
		& = \frac{\left(\prod_{j=1}^nx_j\right)^{n-1}}{\left(\prod_{j=1}^nx_j\right)^{n-2}}\cdot \left({\rm diag}\left(\frac{1}{x_1}, \hdots, \frac{1}{x_n}\right)\right)^{\rm adj}\\
		& = \prod_{j=1}^n x_j \cdot {\rm diag}\left(\prod_{j\neq 1}\frac{1}{x_j}, \hdots,\prod_{j\neq n}\frac{1}{x_j}\right)\\
		& = {\rm diag}\left(x_1, \hdots, x_n\right).
	\end{align*}

For general $f$, we take a factorization of $f$ into irreducible polynomials $f = \prod_k f_k$. 
By Lemma~\ref{lem:factorDelta}, $\Delta_{ij}(f_k)$ is a square for each $i,j,k$ 
and so by the arguments above, $f_k$ has a determinantal representation of the correct form. 
Taking a block diagonal representation of these representations (and permuting the rows and columns if necessary to reorder $x_1, \hdots, x_n$) gives a determinantal representation for $f$. 
\end{proof}

\begin{example}
For $n=4$ and $R = \Z$, we apply this algorithm to the symmetric quartic
\[
f \ = \  x_1 x_2 x_3 x_4 - (x_1x_2 + x_1x_3 + x_1x_4 + x_2x_3+x_2x_4 + x_3 x_4) + 2(x_1+x_2+x_3+x_4) -3.
\]
For each $i\in [4]$, we take $g_{ii}$ to be $\frac{\partial f}{\partial x_i} = \prod_{j\neq i}x_j - \sum_{j\neq i}x_j +2$.
For every $i\neq j$, we find that $\Delta_{ij}(f)  = \frac{\partial f}{\partial x_i}\frac{\partial f}{\partial x_j} - f \frac{\partial^2 f}{\partial x_i \partial x_j}$ equals $(x_k -1)^2(x_{\ell}-1)^2$ where $\{k, \ell\} = [4]\backslash\{i,j\}$. 
For each $j=2, 3, 4$, we can choose $g_{1j} = (1-x_k)(x_{\ell}-1)$ with $\{k, \ell\} = [4]\backslash\{1,j\}$. 
Then for $\{j,k,\ell\} = \{2,3,4\}$, we find that  
$g_{11}(1-x_1)(x_{j}-1) - g_{1k}g_{1\ell}$
equals $(1-x_j)f$ and so we also take $g_{k\ell} =(1-x_1)(x_{j}-1) $. 
We then construct the $4\times 4$ matrix $G = (g_{ij})_{1\leq i, j \leq 4} = $  
	\[
{\tiny 
\begin{pmatrix} x_2x_3x_4 -x_2-x_3-x_4 + 2& -x_3x_4 + x_3 + x_4 -1& -x_2x_4 + x_2 + x_4 -1 & -x_2x_3 + x_2 + x_3 -1  \\  -x_3x_4 + x_3 + x_4 -1 & x_1x_3x_4 -x_1-x_3-x_4 +2& -x_1x_4 + x_1 + x_4 -1 & -x_1x_3 + x_1 + x_3 -1 \\
		-x_2x_4 + x_2 + x_4 -1 & -x_1x_4 + x_1 + x_4 -1 & x_1x_2x_4 -x_1-x_2-x_4 +2 & -x_1x_2 + x_1 + x_2 -1  \\ -x_2x_3 + x_2 + x_3 -1   & -x_1x_3 + x_1 + x_3 -1 & -x_1x_2 + x_1 + x_2 -1 &x_1x_2x_4 -x_1-x_2-x_4 +2 \end{pmatrix}.
		}
	\]
Note that only the diagonal entries of this matrix have degree three, so if we homogenize all entries to 
have degree three and then set the homogenizing variable equal to zero, the result is the diagonal matrix  
$G({\bf x}, 0)  = x_1x_2x_3x_4{\rm diag}(x_1^{-1}, x_2^{-1},x_3^{-1},x_4^{-1})$ appearing in the proof of Theorem~\ref{thm:DeltaDetRep}($\Leftarrow$). 
Moreover, the $2\times 2$ minors of this matrix are divisible by $f$ and so its $3\times 3$ minors
are divisible by $f^2$. 
Taking the adjugate of $G$ and dividing by $f^2$, we find the desired symmetric matrix whose determinant gives the polynomial $f$:
	\[ M = \frac{1}{f^2} G^{\rm adj} = {\small \begin{pmatrix} x_1 & 1 & 1 & 1 \\ 1 &  x_2 & 1 & 1 \\
		1 &1 & x_3 & 1 \\ 1 & 1 & 1 &  x_4 \end{pmatrix}.} \]
\end{example}

Before moving on to applications to the principal minor problem, 
we remark that the rank of the matrices $M_j$ can be recovered from $f$. 
Here we define the rank of a matrix $M\in {\rm Sym}_n(R)$ to be the maximum $r\in \N$
so that there is a non-zero $r\times r$ minor of $M$. 
Note that the rank of $M$ over $R$ is the same as its rank over the field of fractions of $R$.

\begin{lem}\label{lem:rank} Let $R$ be an integral domain. 
If $f = \det({\rm diag}(x_1,\hdots,x_n) + \sum_{j=n+1}^m x_j M_j)$ where $M_j\in {\rm Sym}_n(R)$, then the rank of $M_j$ equals the degree $f$ in the variable $x_j$. 
\end{lem}

\begin{proof}
The bound $\deg_j(f)\leq {\rm rank}(M_j)$ follows from the Laplace expansion of the determinant. 
To see equality, it suffices to take $j=m=n+1$. 
Let $f = \det({\rm diag}(x_1,\hdots,x_n) + yA)$  where $A\in {\rm Sym}_n(R)$. Then $f = \sum_{S\subseteq [n]}A_S{\bf x}^{[n]\backslash S}y^{|S|}$ equals the homogenization of $f_A$. 
From this we see that the degree of $f$ in the variable $y$ equals the 
size of the largest nonzero \emph{principal} minor of $A$. 
By the so-called Principal Minor Theorem \cite[Strong PMT 2.9]{KodiyalamSwan08}, 
this coincides with the size of the largest nonzero minor of $A$, i.e. ${\rm rank}(A)$. 
\end{proof}

Recall that to an element ${\bf a} = (a_S)_{S\subseteq [n]}$ in $R^{2^n}$ we associate the multiaffine polynomial 
\[f_{\bf a} = \sum_{S\subseteq [n]}a_S {\bf x}^{[n]\backslash S}.\] 

\begin{thm}\label{thm:DeltaMinors}
Let $R$ be a unique factorization domain. 
An element ${\bf a} = (a_S)_{S\subseteq [n]}$ in $R^{2^n}$ is in the image of ${\rm Sym}_n(R)$ under
the principal minor map if and only if $a_{\emptyset}=1$ and for every $i,j \in [n]$,
$\Delta_{ij}(f_{\bf a})$ is a square in $R[{\bf x}]$. 
\end{thm}

\begin{proof}
By Corollary~\ref{cor:SquareDelInvariance}, 
$\Delta_{ij}(f_{\bf a})$ is a square in $R[{\bf x}]$ if and only if 
$\Delta_{ij}(\overline{f_{\bf a}})$ is a square in $R[{\bf x},y]$.
Furthermore, by Theorem~\ref{thm:DeltaDetRep}, $\Delta_{ij}(\overline{f_{\bf a}})$ is a square in $R[{\bf x}, y]$
 if and only if there exists a symmetric matrix $A\in {\rm Sym}_n(R)$ for which 
 $\overline{f_{\bf a}} =  \det({\rm diag}(x_1,\hdots,x_n) + y A)$. 
\end{proof}

\begin{cor}
Let $R$ be a unique factorization domain. Then the image of ${\rm Sym}_n(R)$ under the principal minor map is invariant under the action of 
$G^n \rtimes S_n$, where $G$ is the subgroup of $\SL_2(R)$ defined by $G=\left\{\begin{pmatrix} 1 & r \\ 0 & 1\end{pmatrix}\mid r \in R\right\}$.
\end{cor}
\begin{proof} 
Let ${\bf a} = (a_S)_{S\subseteq[n]} $ be an element in the image of the principal minor map and let 
$f_{\bf a}$ be the associated multiaffine polynomial. 
Consider ${\bf b} = (b_S)_{S\subseteq [n]}$ with ${\bf b} = \gamma \cdot {\bf a}$ where $\gamma \in G^n\rtimes S_n$. 
Theorem \ref{thm:DeltaDetRep} implies that $a_{\emptyset} = 1$ and for every $i,j,\, \Delta_{ij}(f_{\bf a})$ is a square in $R[{\bf x}]$. Corollary \ref{cor:SquareDelInvariance} implies that for every $i,j,\, \Delta_{ij}(f_{\bf b})$ is a square in $R[{\bf x}]$. Using Theorem~\ref{thm:DeltaDetRep} again, it is enough to show that $b_{\emptyset}$, the coefficient of $x_1\cdots x_n$ in $f_{\bf b}$, equals one. 
It is clear that this coefficient is invariant under the action of $S_n$, and so 
it suffices to take $\gamma = (\gamma_1,\ldots,\gamma_n)\in G^n$ where $\gamma_i = \tiny{\begin{pmatrix}
1 & r_i \\ 0 & 1 \end{pmatrix}}$, for which 
	\[f_{\bf b} = \gamma \cdot f_{\bf a} = f_{\bf a}(x_1 + r_1,\ldots, x_n + r_n).\]
From this, we see that coefficient of $x_1\cdots x_n$ in $f_{\bf b}$ is equal to one. 
\end{proof}

The subgroup $G$ is the maximal subgroup that preserves the leading 
coefficients of polynomials $f_A$.  Consider an element of $\gamma = {\tiny \begin{pmatrix}a  & b\\ c & d\end{pmatrix}}\in \SL_2(R)$ acting on the first coordinate $x_1$ of $f_A$. 
Then  
\[
{\rm coeff}(\gamma\cdot f_A, x_1\cdots x_n)   = {\rm coeff}\left((ax_1+b)\frac{\partial f_{A}}{\partial x_1} + (cx_1 + d)(f_A|_{x_1=0}), x_1\cdots x_n\right) = aA_{\emptyset} + c A_1.
\]
In order to preserve the coefficient of $x_1\cdots x_n$, $\gamma$ must satisfy 
$1= a + c A_1$ for every value of $A_1$. This implies that $c=0$ and $a=1$. 
The condition  $ad - bc = 1$ then implies that $d=1$.

\begin{cor}\label{cor:SL2_scaled}
Let $R$ be a unique factorization domain and let ${\bf a}$ be the vector of principal minors of a matrix in ${\rm Sym}_n(R)$. 
Fix $\gamma\in  \SL_2(R)^n\rtimes S_n$ and let $\lambda$ denote the coefficient of $\prod_{i=1}^nx_i$ in $\gamma \cdot f_{\bf a}$. If  $\lambda \neq 0$, then for some $A\in {\rm Sym}_n(R)$,
\[
\gamma\cdot f_{\bf a} = \lambda \cdot \det\left({\rm diag}(x_1, \hdots, x_n) + \frac{1}{\lambda} A\right).
\] 
That is, $\frac{1}{\lambda} \gamma \cdot {\bf a}$  belongs the image of the principal minor map over $\frac{1}{\lambda}R$. 
\end{cor}
\begin{proof}
	Let ${\bf b} = \frac{1}{\lambda} \gamma \cdot {\bf a}$ and let $f_{\bf b} $ be the multiaffine polynomial 
	associated to ${\bf b}$. 
	Then ${\rm coeff}\left(f_{\bf b} , \prod_{i} x_i \right)= \frac{1}{\lambda} \lambda = 1$ 
	and by Corollary~\ref{cor:SquareDelInvariance}, 
	$\Delta_{ij}f_{\bf b} = \frac{1}{\lambda^2}\Delta_{ij}(\gamma\cdot f_{\bf a})$ is a square for every $i,j$. 
	Hence, by Theorem~\ref{thm:DeltaDetRep}, there exists a symmetric matrix $B$ with entries in 
	$R(\frac{1}{\lambda})$ with
	\[
	f_{\bf b} = \det\left({\rm diag}(x_1, \hdots, x_n) + B\right).
	\]
	We claim that $\lambda B$ has entries in $R$. To see this, note that 
	${\bf b} = \frac{1}{\lambda} \gamma \cdot {\bf a}$ is the vector of principal minors of $B$ 
	and that the entries of $\gamma \cdot {\bf a}$ belong to $R$. So all principal minors of $B$ 
	belong to $\lambda^{-1}R$. This immediately shows that the diagonal elements of $\lambda B$ 
	belong to $R$. 
	
	For the off-diagonal elements, fix $i\neq j\in [n]$ and let $z$ denote the $(i,j)$th entry of $B$. 
	Then $b_i b_j - b_{ij} = z^2$, where ${\bf b} = (b_S)_{S\subseteq [n]}$, 
	which implies that $\lambda^2z^2 = (\lambda b_i)(\lambda b_j) -\lambda^2b_{ij}  \in R$. 
	By construction, $z\in R(1/\lambda)$ so we can take the minimal $m\in \N$ for which 
	$\lambda^mz \in R$. So $\lambda^mz = r\in R$ where either $m=0$ or $\lambda$ does not divide $r$. 
	Then 
	$\lambda^{2m-2}(\lambda^2z^2) = (\lambda^mz)^2= r^2$.
	Since $\lambda^2z^2\in R$, we see that $\lambda^{2m-2}$ divides $r^2$.
	If $m>1$, this contradicts the assumption that $\lambda$ does not divide $r$. 
	Therefore $A = \lambda B \in {\rm Sym}_n(R)$, as desired.
\end{proof}

\begin{example}
	For $R = \Z$ and $n=2$ consider the matrix $A = \tiny{\begin{pmatrix} 2 & 1 \\ 1 & 1 \end{pmatrix}}$ and its 
	vector of principal minors ${\bf a}=(1,2,1,1)$, giving 	$f_{\bf a} = x_1 x_2 + x_1 + 2x_2 + 1$. 
	The image of $f_{\bf a}$ under the action of  $\gamma = \left(\tiny{\begin{pmatrix} 2 & -1 \\ 1 & 0 \end{pmatrix}}, 
						\tiny{\begin{pmatrix}	1 & 0 \\ 0 & 1 \end{pmatrix}}\right) \in \SL_2(R)^2$  is
							\[\gamma \cdot f_{\bf a} = 4 x_1 x_2 + 3x_1 - x_2 -1.\]					
	Since $4\neq 1$, the vector $(4,-1,3,-1)$ is not the vector of principal minors of any symmetric matrix over $\Z$. 
	However, as promised by Corollary~\ref{cor:SL2_scaled}, the vector
	 $\frac{1}{4}\gamma \cdot {\bf a} = \left(1,\frac{-1}{4},\frac{3}{4},\frac{-1}{4}\right)$ is the vector of principal minors of 
	 the matrix $B = \frac{1}{4}{\tiny{\begin{pmatrix} -1& 1 \\ 1 & 3 \end{pmatrix}}}$ over $\frac{1}{4}\Z$.		
\end{example}

\begin{cor}
Let $R=\F$ be an infinite field. The Zariski closure of the image of the principal minor map in $\P^{2^n-1}(\F)$ is invariant under the action of $\SL_2(\F)^n \rtimes S_n$.
\end{cor}

\begin{proof} The image is immediately invariant under the action of $S_n$, so it 
suffices to show invariance under $\SL_2(\F)^n$.  Using the action of $S_n$, it suffices to 
show this by acting with $\SL_2(\F)$ on the first coordinate. 
Let $A\in {\rm Sym}_n(\F)$, giving a point $\varphi(A) = (A_S)_{S}$ in the image of the principal minor map. 
Consider the open subset 
\begin{align*} 
U & = \left\{\gamma \in \SL_2(\F) : {\rm coeff}( \gamma \cdot f_A, x_1\cdots x_n) \neq 0\right\} \\
&= \left\{\begin{pmatrix}a & b\\ c & d\end{pmatrix} \in \F^{2\times 2}: ad-bc=1 \ \text{ and } \
aA_{\emptyset}+  cA_1\neq 0\right\}.
\end{align*}
By Corollary~\ref{cor:SL2_scaled}, for every $\gamma\in U$, 
$\gamma\cdot \varphi(A) $ belongs to the image of the principal minor map, up to scaling.  
The parametrization $(a,b,c)\dashmapsto (a,b,c,(1+bc)/a)$ shows that 
$\SL_2(\F)$ is a rational variety over $\F$.
Since $\F$ is infinite, the set of $(a,b,c)\in \F^3$ such that $a\neq 0$ and $aA_{\emptyset}+  cA_1\neq 0$ 
is Zariski-dense in $\F^3$. It follows that $U$ is Zariski-dense in $\SL_2(\F)$. 
Since $\gamma \mapsto \gamma\cdot \varphi(A)$ defines a rational map 
$\SL_2(\F) \to \P^{2^n-1}(\F)$, it follows that for 
any $\gamma\in \SL_2(\F)$, $\gamma\cdot \varphi(A) $ 
belongs to the Zariski-closure of the image of $\varphi$ in $\P^{2^n-1}(\F)$.
\end{proof}

\section{Defining the set of multiquadratic squares}\label{sec:multiQsquares}

The polynomials $\Delta_{ij}(f)$ appearing in Theorem~\ref{thm:DeltaMinors} have degree $\leq$ two in
each variable. In order to make use of this characterization, in this section we 
find algebraic conditions characterizing the set of squares in $R[{\bf x}]_{\leq {\bf 2}}$. 
In fact, to simplify the notation used in the arguments below we consider the multihomogenezations, 
as in Proposition~\ref{prop:homDelta}, and characterize the set of squares in $R[{\bf x},{\bf y}]_{{\bf 2}}$.

In a slight abuse of notation, we define $\P^1(R)$ to be the following subset of $R^2$: 
\[\P^1(R) = \{(r,1) : r\in R\}\cup \{(1,0)\}.\] 

\begin{lem}\label{lem:evalZero}
	Let $g\in R[{\bf  x}, {\bf y}]_{{\bf d}}$ where ${\bf d} = (d_1, \hdots, d_n)\in \N^n$.  
	Let $P_i\subseteq \P^1(R)$ be a set of size $d_i+1$. 
	Then $g$ is the zero polynomial if and only if $g(p) = 0$ for all $p\in P_1\times \cdots \times P_n$. 
\end{lem}

\begin{proof}
	We prove this by induction on $n$. 
	Note that for $n=1$, a polynomial $g\in R[x,y]_{d}$ vanishes at $(a,b)\in \P^1(R)$ 
	if and only if $bx-ay$ divides $g$. 
	Therefore a bivariate form $g\in R[x,y]_{d}$ cannot have more than $d$ roots in $\P^1(R)$. 
	Now suppose that $n\geq  1$.  Fix $(a,b)\in P_{n+1}$. 
	The polynomial $g({\bf x}, a,{\bf y}, b)$ vanishes on 
	$P_1\times \cdots \times P_{n}$ and is therefore 
	identically zero by induction. 
	This means that, considered as a bivariate form in $x_{n+1}, y_{n+1}$ over the ring $R[{\bf x},  {\bf y}]$, 
	$g$ vanishes at the $d_{n+1}+1$ points in $P_{n+1}$ and it is therefore identically zero in 
	$R[{\bf x}, {\bf y}][x_{n+1},y_{n+1}] = R[{\bf x}, x_{n+1}, {\bf y}, y_{n+1}]$.
\end{proof}

\begin{lem}\label{lem:factorSquare}
	A polynomial $g({\bf x},s,{\bf y},t) = g_2({\bf x},{\bf y})s^2 + g_1({\bf x},{\bf y}) s t + g_0({\bf x},{\bf y})t^2  \in R[{\bf x},s,{\bf y},t]_{{\bf 2}}$ 
	is a square in $R[{\bf x},s,{\bf y},t]$ if and only if 
	$g_0({\bf x},{\bf y})$ and $g_2({\bf x},{\bf y})$ are squares in $R[{\bf x},{\bf y}]$ and ${\rm Discr}_{(s,t)}(g) =0$.
\end{lem}
\begin{proof}
	Suppose that  $g_2 = (h_2)^2$ and $g_0= (h_0)^2$, where $h_2, h_0\in R[{\bf x},{\bf y}]$ and 
	\[
	{\rm Discr}_{(s,t)}(g) =  g_1^2 - 4 g_0 g_2 = 0  \ \ \text{ in } \ \ R[{\bf x},{\bf y}]. 
	\] 
	Then $4g_0g_2= (2h_0h_2)^2 = g_1^2 $, giving that 
	$ (2h_0h_2)^2 - g_1^2 = (2h_0h_2- g_1)(2h_0h_2+ g_1)=0$ in $R[{\bf x},{\bf y}]$.  
	It follows that $g_1 = \pm 2h_0h_2$. Changing the sign of $h_0$ if necessary, we can assume that $g_1 = 2h_0h_2$.
	Then 
	\[
	g = (h_2)^2s^2 +2 h_0h_2  s t + (h_0)^2t^2
	= (s h_2 + th_0)^2
	\]
	is a square, as desired. 
	Conversely, if $g =(s h_2 + th_0)^2$, we see that $g_0= (h_0)^2$ and $g_2 = (h_2)^2$ and 
	${\rm Discr}_{(s,t)}(g)=0$. 
\end{proof}

\begin{thm}\label{thm:square} 
	Let $R$ be a unique factorization domain with $|R|\neq 3$. 
	A multiquadratic polynomial 
	$g = \sum_{\alpha \in \{0,1,2\}^n}c_{\alpha}{\bf x}^{\alpha}{\bf y}^{{\bf 2}-\alpha}$
	is a square, i.e. $g = h^2$ with $h\in R[{\bf x}, {\bf y}]_{{\bf 1}}$,
	if and only if  for every $\beta\in \{0,1\}^n$, $c_{2\beta}$ is a square in $R$ 
	and $ {\bf c} = (c_{\alpha})_{\alpha \in \{0,1,2\}^n}$ satisfies 
	the images of 
	\begin{equation}\label{eq:squareC}
		(c_{(1,{\bf 0})})^2-4c_{(0,{\bf 0})} c_{(2,{\bf 0})}=0
	\end{equation}
	under the action of $\SL_2(R')^n \rtimes S_n$ where $R'$ is any nontrivial subring of $R$ with $1_{R} = 1_{R'}$ 
	 and size $\geq 4$ whenever ${\rm char}(R)\neq 2$.
\end{thm}

\begin{proof}
	($\Rightarrow$) If $g = h^2$, then for every $\gamma\in \SL_2(R)^n \rtimes S_n$, 
	$\gamma\cdot g = (\gamma\cdot h)^2$.  
	Note here that the action on $g$ comes from the action on 
	$R[{\bf x},{\bf y}]_{{\bf 2}}$ and the action on $h$ comes from the action on $R[{\bf x},{\bf y}]_{{\bf 1}}$. 
	The specialization of $\gamma\cdot g$ to 
	$({\bf x},{\bf y}) = (x_1, {\bf 0},y_1,{\bf 1})$ will be a square in $R[x_1,y_1]$. 
	Moreover, the coefficient $c_{\beta}$ of ${\bf x}^{2\beta}{\bf y}^{{\bf 2}-2\beta}$ in $g$ is
	the square of the coefficient ${\bf x}^{\beta}{\bf y}^{{\bf 1}-\beta}$ in $h$, 
	and in particular the square of an element in the ring $R$.

	($\Leftarrow$)
	We prove this by induction on $n$. 
	This holds immediately for $n=1$ by Lemma~\ref{lem:factorSquare}. 
	Now suppose $n\geq 1$ and take $g= g_2s^2 + g_1st +g_0t^2   \in R[{\bf x}, s,{\bf y}, t]$.
	Fixing $n+1$ in $S_{n+1}$ and ${\tiny\begin{pmatrix} 1 & 0 \\ 0 & 1\end{pmatrix}}$ or 
	${\tiny\begin{pmatrix} 0 & -1 \\ 1 & 0\end{pmatrix}}$ for the $(n+1)$st coordinate in
	$(\SL_2(R))^{n+1}$, we see that both $g_0$ and $g_2$ 
	satisfy the hypothesis of the theorem and so, by induction, are squares in $R[{\bf x},{\bf y}]$. 
	Here $1$ denotes the common multiplicative identity of $R$ and $R'$. 
	
	If ${\rm char}(R)\neq 2$ then $|R'|\geq 4$ and we can take a set $P \subset \P^1(R')$ of size five. 
	Define a map $\varphi: \P^1(R')\to {\rm SL}_2(R')$ 
	by $\varphi((1,0)) = \tiny{\begin{pmatrix} 0 & 1  \\ -1 & 0 \end{pmatrix}}$ and 
	for $r\in R'$, $\varphi((r,1)) = \tiny{\begin{pmatrix} 1 & r  \\ 0 & 1 \end{pmatrix}}$.  
	Then to $(p_1, \hdots, p_n)\in P^n$, we associate the element 
	$\gamma  = \left(\varphi(p_1), \ldots,\varphi(p_n), {\rm Id}_2\right)$ of $\SL_2(R')^{n+1}$.
	Acting on $g$ by $\gamma$ and then specializing to $x_1=\hdots= x_n=0$ and $y_1=\hdots =y_n=1$ gives 
	\begin{equation}\label{eq:SL_2restriction}
		\bigl(\gamma \cdot g\bigl)|_{{\bf x}=0, {\bf y}=1} = g({\bf a}, s, {\bf b}, t)
		\ \ \text{ and  } \ \ 
		\bigl({\rm Discr}_{(s,t)}(\gamma\cdot g)\bigl)|_{{\bf x}=0, {\bf y}=1} = 
		\bigl({\rm Discr}_{(s,t)}g\bigl)|_{{\bf x} =  {\bf a}, {\bf y}={\bf b}}
	\end{equation}
	where $p_i = (a_i, b_i)$ and ${\bf a} = (a_1, \hdots, a_n)$,  ${\bf b} = (b_1, \hdots, b_n)$. 
	Transposing $(x_1,y_1)$ and $(s,t)$ using the action of $S_{n+1}$, we see that by assumption, 
	this evaluation of the discriminant of $\gamma\cdot g$ must be zero. 
	Since the discriminant has degree $\leq 4$ in each variable, Lemma~\ref{lem:evalZero}
	implies that it is identically zero. 
	Then by Lemma~\ref{lem:factorSquare}, $g$ is a square in $R[{\bf x}, s, {\bf y}, t]$. 
	
	If ${\rm char}(R) = 2$, the discriminant ${\rm Discr}_{(s,t)}g$ simplifies to a square, $g_1^2$, which 
	by \eqref{eq:SL_2restriction}, must vanish at the points $\{(1,0),(1,1),(0,1)\}^n\subseteq (\P^1(R))^n$.  
	Since $g_1$ has degree $2$ in each variable and must vanish at these points, 
	Lemma~\ref{lem:evalZero} implies that $g_1$ is identically zero. 
	Therefore by Lemma~\ref{lem:factorSquare}, $g$ is a square. 
\end{proof}

\begin{remark}\label{rem:SL2subset}
	The proof of Theorem~\ref{thm:square} reveals that only a small subset of ${\rm SL}_2(R)$ is needed 
	in each coordinate to characterize multiquadratic squares, specifically a set of size five.
\end{remark}

Up to isomorphism, there is only one ring of size three, namely $\F_3$.  
The exclusion of $\F_3$ in the statement of Theorem~\ref{thm:square} is a necessary one, 
as the following example demonstrates. 
\begin{example} \label{ex:F3weird}
For $R = \F_3$ and $n=3$ consider the multiquadratic form 
	\begin{align*}
		g \ = \ &  x_1^2 (x_2 y_3 - x_3 y_2)^2 - 2 x_1 y_1 (x_2 y_3 + x_3 y_2) (x_2 x_3 - y_2 y_3) + 
		y_1^2 (x_2 x_3 + y_2 y_3)^2
		\\
		= \ &  x_1^2 x_2^2 y_3^2+ x_1^2 x_3^2 y_2^2+ x_2^2 x_3^2 y_1^2 
		- 2 x_1^2 x_2 x_3 y_2 y_3 - 2 x_1 x_2^2 x_3 y_1 y_3  - 2 x_1 x_2 x_3^2 y_1 y_2\\
		&+ 2 x_1x_2y_1y_2y_3^2 +   2 x_1x_3y_1y_2^2y_3+   2 x_2x_3y_1^2y_2y_3 
		+y_1^2y_2^2y_3^2.
	\end{align*}
	The discriminant of $g$ with respect to $(x_1, y_1)$ is given by 
	\[
	{\rm Discr}_{(x_1, y_1)}g = 16 x_2y_2(x_2+y_2)(x_2-y_2) x_3y_3(x_3+y_3)(x_3-y_3).
	\]
	This polynomial is non-zero in $\F_3[{\bf x},{\bf y}]_{\bf 4}$ but vanishes on all points in $(\P^1(\F_3))^2$. 
	The polynomial $g$ is invariant under permutations of indices, so the discriminants ${\rm Discr}_{(x_2, y_2)}g$ 
	and ${\rm Discr}_{(x_3, y_3)}g $ have the same property. 
\end{example}

Before using this result to characterize the image of the principal minor map, we 
record a few of the notable special cases of Theorem~\ref{thm:square}. 
In particular, over $\C$, the set of multiquadratic squares is 
defined by the orbit of a single polynomial and the only additional constraint 
over $\R$ is the nonnegativity of the orbit of one other polynomial.

\begin{cor}\label{cor:squareRC} 
	Let $g = \sum_{\alpha \in \{0,1,2\}^n}c_{\alpha}{\bf x}^{\alpha} \in \C[{\bf x}]$. 
	The polynomial $g$ is a square, $g = h^2$ with $h\in \C[{\bf x}]_{\leq {\bf 1}}$
	if and only if ${\bf c} = (c_{\alpha})_{\alpha \in \{0,1,2\}^n}$ satisfies 
	the images of 
	\begin{equation}\label{eq:squareC}
		(c_{(1,{\bf 0})})^2-4c_{(0,{\bf 0})} c_{(2,{\bf 0})}=0
	\end{equation}
	under the action of $\SL_2(\C)^n \rtimes S_n$.
	Moreover, $g$ is a square over $\R$, 
	$g = h^2$ with $h\in \R[{\bf x}]_{\leq {\bf 1}}$, if and only if 
	its coefficients ${\bf c} = (c_{\alpha})_{\alpha \in \{0,1,2\}^n}$ satisfy the images of
	\begin{equation}\label{eq:squareR}
		(c_{(1,{\bf 0})})^2-4c_{(0,{\bf 0})} c_{(2,{\bf 0})}=0 
		\ \ \ \text{ and } \ \ \ 
		c_{(0,{\bf 0})} \geq0.
	\end{equation}
	under the action of $\SL_2(\R)^n \rtimes S_n$.
\end{cor}

As seen in the proof of Theorem~\ref{thm:square}, in characteristic two, 
linear equations in the coefficients suffice to cut out of the set of squares in $R[{\bf x}]_{\leq {\bf 2}}$.

\begin{cor}\label{cor:squareF2} 
	Let $g = \sum_{\alpha \in \{0,1,2\}^n}c_{\alpha}{\bf x}^{\alpha} \in R[{\bf x}]$ where $R$ 
	is a UFD of characteristic two.  
	The polynomial $g$ is a square, $g = h^2$ with $h\in R[{\bf x}]_{\leq {\bf 1}}$
	if and only if ${\bf c} = (c_{\alpha})_{\alpha \in \{0,1,2\}^n}$ satisfy 
\begin{itemize}
\item[(i)] for every $\beta\in \{0,1\}^n$, $c_{2\beta} $ is a square in $R$, and  
\item[(ii)] for every $\gamma\in \SL_2(\F_2)^n \rtimes S_n$, $(\gamma\cdot {\bf c})_{(1,{\bf 0})} =0$.
\end{itemize}
\end{cor}

\section{Characterizing the image of the principal minor map}\label{sec:finale}

We can now combine the characterization of multiaffine determinantal polynomials from Section~\ref{sec:DetRep} 
and the characterization of multiquadratic squares in Section~\ref{sec:multiQsquares} to give a complete description 
of the principal minor map over any unique factorization domain of size $\neq 3$. 

\begin{thm}\label{thm:HypDet2} Let $R$ be an UFD with $|R|\neq 3$ and 
let ${\bf a} = (a_S)_{S\subseteq [n]}\in R^{2^n}$ with $a_{\emptyset}=1$. There exists a symmetric matrix over $R$ with principal minors ${\bf a}$ if and only if 
\begin{itemize}
\item[(i)] for every $i,j\in [n]$,  $a_{i}a_{j} - a_{ ij} $ is a square in $R$, and  
\item[(ii)] for every $\gamma\in \SL_2(R)^n \rtimes S_n$, $(\gamma\cdot{\rm HypDet})({\bf a}) =0$.
\end{itemize}
\end{thm}
\begin{proof}
($\Rightarrow$) 
If ${\bf a}$ belongs to the image of the principal minor map over $R$, 
then $a_ia_j - a_{ij}$ is the square of the $(i,j)$th entry of the representing matrix $A$, and so is a square in $R$. 
Also, by Theorem~\ref{thm:DeltaDetRep} $\Delta_{ij}f_{\bf a}$ is a square for all $i,j\in [n]$.
Then, by Corollary~\ref{cor:SquareDelInvariance}, for any $\gamma \in \SL_2(R)^n \rtimes S_n$, 
$\Delta_{ij}(\gamma \cdot f_{\bf a}) $ is  square. In particular, 
\[ 
{\rm HypDet}(\gamma\cdot {\bf a}) = {\rm Discr}_{x_3}(\Delta_{12}(\gamma \cdot f_{\bf a}))|_{x_4=\hdots = x_n=0} = 0. 
\]

($\Leftarrow$) 
First, we show that ${\bf a}$ belongs to the image of the principal minor map 
over the algebraic closure of the fraction field of $R$, $\F = \overline{({\rm frac}(R))}^{alg}$.
Then $R$ is a subring of $\F$ and has size $\geq 4$ if ${\rm char}(R) = {\rm char}(\F) \neq 2$. 
Let $f = f_{\bf a}$. 
Every element of $R$ is a square over $\F$. 
Then by Theorem~\ref{thm:square}, 
$\Delta_{ij}f =  \sum_{\alpha \in \{0,1,2\}^n}c_{\alpha}{\bf x}^{\alpha}$ is a square in $\F[{\bf x}]$ 
if and only if $ {\bf c} = (c_{\alpha})_{\alpha \in \{0,1,2\}^n}$ and its images under 
$\SL_2(R)^n \rtimes S_n$ satisfy $(c_{(1,{\bf 0})})^2-4c_{(0,{\bf 0})} c_{(2,{\bf 0})}=0$. 
This condition for all $i,j\in [n]$ is equivalent to the condition that $(\gamma\cdot {\rm HypDet})({\bf a}) =0$ 
for all  $\gamma\in \SL_2(R)^n \rtimes S_n$. Therefore $\Delta_{ij}f$ is a square in $\F[{\bf x}]$ for all $i,j$. 
By Theorem~\ref{thm:DeltaMinors}, ${\bf a}$ is the vector of principal minors of some matrix $A$ with entries in $\F$. 
The diagonal entries of $A$ are entries $a_i$ in this vector and therefore belong to $R$. 
Let $z$ denote the $(i,j)$th entry of $A$ for some $i\neq j$. 
By assumption $z^2 = a_i a_j - a_{ij} = r^2$ for some $r\in R$. Then $(z+r)(z-r)=0$ 
implying that $z=\pm r\in R$. Therefore $A\in {\rm Sym}_n(R)$.  
\end{proof}

\begin{remark}\label{rem:count}
Note that in part (ii) of the characterization in Theorem~\ref{thm:HypDet2}, it suffices to take $\binom{n}{3}\cdot 5^{n-3}$ elements  $\gamma\in \SL_2(R)^n \rtimes S_n$. Specifically, let $P\subseteq R$ be a set of size 5.  For any subset $\{i,j,k\} \subseteq[n]$ and point ${\bf p}\in P^{n-3}$, we get an equation 
\[
{\rm Discr}_{x_k}(\Delta_{ij}(f_{\bf a}))|_{{\bf x} = {\bf p}} = 0. 
\]
This is enough to ensure that ${\rm Discr}_{x_k}(\Delta_{ij}(f_{\bf a}))$ is identically zero. 
More generally, we can take evaluations of ${\rm Discr}_{(x_k,y_k)}(\Delta_{ij}(f_{\bf a}))^{{\bf 2}-{\rm hom}}$ at 
$5^{n-3}$ points in $\P^1(R)$, as in Section~\ref{sec:multiQsquares}.  
While this expression appears to depend on the ordering of $i,j,k$, one can check that for any $f\in R[{\bf x}]_{\leq {\bf 1}}$, 
${\rm Discr}_{x_k}(\Delta_{ij}(f)) = {\rm Discr}_{x_j}(\Delta_{ik}(f)) = {\rm Discr}_{x_i}(\Delta_{jk}(f))$. See, for example, the proof of Theorem~3.1 in \cite{WagnerSurvey}. 
When $R$ is a ring of characteristic 2, it suffices to take $P$ to have size $3$, giving a total of $\binom{n}{3}\cdot 3^{n-3}$ equations. 
\end{remark}

Applying this to $R = \C$, $\R$, and $\F_2$, we find the following immediate consequences. 

\begin{cor}\label{cor:HypDetC}
Let ${\bf a} = (a_S)_{S\subseteq [n]}\in \mathbb{C}^{2^n}$ with $a_{\emptyset}=1$. There exists a symmetric matrix over $\C$ with principal minors ${\bf a} $ if and only if ${\bf a}$ and all its images 
under the action of $\SL_2(\mathbb{C})^n \rtimes S_n$ satisfy the $2 \times 2 \times 2$ hyperdeterminant ${\rm HypDet}({\bf a} ) =0$. 
\end{cor}

\begin{cor}\label{cor:HypDetR}
Let ${\bf a} = (a_S)_{S\subseteq [n]}\in \mathbb{R}^{2^n}$ with $a_{\emptyset}=1$. 
 There exists a symmetric matrix over $\R$ with principal minors ${\bf a} $
if and only if ${\bf a} $ and all its images under the action of $\SL_2(\mathbb{R})^n \rtimes S_n$
satisfy
\[{\rm HypDet}({\bf a}) =0 \text{ and } a_{1} a_{2} - a_{\emptyset} a_{12} \geq0.\]
\end{cor}

\begin{cor}\label{cor:HypDetF2}
Let ${\bf a} = (a_S)_{S\subseteq [n]}\in R^{2^n}$ with $a_{\emptyset}=1$ where $R$ has characteristic two. 
 There exists a symmetric matrix over $R$ with principal minors ${\bf a} $
if and only if 
\begin{itemize}
\item[(i)] for every $i,j\in [n]$, $a_{i}a_{j} - a_{ij} $ is a square in $R$, and  
\item[(ii)] for every $\gamma\in \SL_2(\F_2)^n \rtimes S_n$, $\gamma\cdot(a_{\emptyset} a_{123} +  a_{1} a_{23} +a_{2} a_{13}+a_{3} a_{12})=0$.
\end{itemize}
In particular, for $R = \F_2$, {\rm (i)} is always satisfied and the image of the principal minor map is cut out by the quadratic equations in {\rm (ii)}. 
\end{cor}

It is unclear whether or not Theorem~\ref{thm:HypDet2}  can be extended to $R = \F_3$. 
Example~\ref{ex:F3weird} shows that this would likely require a different proof technique. 
Interestingly, the polynomial $g$ in this example is of the form $\Delta_{12}(f)$ for some 
$f\in \F_3[x_1, \hdots, x_5]_{\leq {\bf 1}}$, but for all such $f$ we have found, the discriminant of 
some other $\Delta_{ij}(f)$ fails to vanish on $(\P^1(\F_3))^{2}$. 

\begin{question}
Does the equivalence in Theorem~\ref{thm:HypDet2} hold for $R = \F_3$? 
\end{question}

\section{Other determinantal representations and connections to $\Gr_\F(d,n)$}\label{sec:otherDetRep}
\subsection{Other multiaffine determinantal representations}
In this section we restrict ourselves to fields and consider the set of multiaffine determinantal polynomials. 
Formally, let $\F$ be an arbitrary field. We call a polynomial $f \in \F[{\bf x}]_{\leq {\bf 1}}$ \emph{determinantal} 
if it can be written in the form 
\begin{equation}\label{eq:det}
f({\bf x}) = \lambda \det\left( V {\rm diag}(x_1, \hdots, x_n) V^T + W\right) =\lambda \det\left(\sum_{i=1}^n x_i v_iv_i^T + W\right)
\end{equation}
for some $\lambda\in \F$, some matrix $V = (v_1, \hdots, v_n)\in \F^{m\times n}$ and some  $W\in {\rm Sym}_m(\F)$ for some $m$. 
Note that when we take $V$ to be the $n\times n$ identity matrix, 
this is exactly the principal minor polynomial $f_W$. 
When $m<n$, the coefficient of $x_1\cdots x_n$ in $f$ is necessarily zero.

\begin{thm}\label{thm:GenDet}
A polynomial $f\in \F[{\bf x}]_{\leq {\bf 1}}$ has a determinantal representation \eqref{eq:det} if and only if 
for all $i,j\in [n]$, $\Delta_{ij}f$ is a square in $\F[{\bf x}]$. 
Moreover, one can always take a representation of size $m = \deg(f)$ in  \eqref{eq:det}. 
\end{thm}

\begin{proof}
($\Rightarrow$) 
Without loss of generality, we show that $\Delta_{12}(f)$ is a square. 
First suppose $v_1$ and $v_2$ are linearly dependent, i.e.~let $v_1 = \alpha v_2$ for some $\alpha\in \F$. 
Then $v_1v_1^T = \alpha^2v_2v_2^T$ and $f(x_1, \hdots, x_n) = f(0, \alpha^2x_1 +x_2, x_3, \hdots, x_n)$. Taking partial derivatives shows that $\frac{\partial f}{\partial x_1} =  \alpha^2\frac{\partial f}{\partial x_2}$ and that $\frac{\partial^2 f}{\partial x_1\partial x_2} = 0$.  Therefore $\Delta_{12}(f) = \alpha^2 (\frac{\partial f}{\partial x_2})^2$
and so is a square.
 
If $v_1$ and $v_2$ are linearly independent, 
then there is an invertible matrix $U$ with $Uv_1 = e_1$ and 
$Uv_2 = e_2$. 
Then 
\[
\det(U)^{2}f = 
\lambda\det\left(
U\left(\sum_{i=1}^n x_iv_iv_i^T + W\right)U^{T}\right)
= \lambda\det\left(
{\rm diag}(x_1,x_2, {\bf 0}) + 
\sum_{i=3}^n x_i\widetilde{v_i}\widetilde{v_i}^T  + \widetilde{W}\right).
\]
where $\widetilde{v}_i = Uv_i$ and $\widetilde{W} = UWU^{T}$.
These matrices are still symmetric and so by 
equation \eqref{eq:dodgson}, $\Delta_{12}(f)$ is a square. 

($\Leftarrow$)
Let $d = \deg(f)$. 
We can assume, without loss of generality, that the coefficient  of $x_1\cdots x_d$ in $f$ is nonzero. 
Moreover since the set of polynomials of the form  \eqref{eq:det} is invariant under scaling, we can assume
that this coefficient equals one. 
Let $\overline{f} \in \F[{\bf x}, y]$ denote the homogenezation of $f$ to total degree $d$. 
By Theorem~\ref{thm:DeltaDetRep}, 
there are matrices $M_{d+1}, \hdots, M_{n+1}$ in ${\rm Sym}_d(\F)$ 
so that
\[\overline{f} = \det\left({\rm diag}(x_1, \hdots, x_d) + \sum_{i=d+1}^nx_iM_i + yM_{n+1}\right).\]  
We take $W = M_{n+1}$. 
It remains to show that for each $i={d+1}, \hdots, n$, the matrix 
$M_i$ has the form $v_iv_i^T$ for some $v_i\in \F^d$. 
Without loss of generality we do this for $i=d+1$. 
Let $g$ denote the specialization of $\overline{f}$
to $x_{k}=0$ for $k=d+2, \hdots, n$ and $y=0$. 
Note that \[g = \sum_{S\subseteq [d]}(M_{d+1})_S(x_{d+1})^{|S|}\prod_{j\in [d]\backslash S}x_j.\]  
However $f$ has degree $\leq 1$ in $x_{d+1}$, and thus so does $g$. 
Therefore by Lemma~\ref{lem:rank}, $M_{d+1}$ has rank $\leq$ one. 
To examine its diagonal entries $(M_{d+1})_i$, 
note that for every $i=1, \hdots, d$, $\Delta_{i(d+1)}g$ is a square. Moreover, the restriction of $g$ to $x_i = x_{d+1}=0$ is 
identically zero, showing that  
\begin{align*} \Delta_{i(d+1)}g &= g_i^{d+1}g_{d+1}^{i}-g^{i(d+1)}g_{i(d+1)}= g_i^{d+1}g_{d+1}^{i}\\
&= 
\biggl((M_{d+1})_{\emptyset}\prod_{j\in [d]\backslash i}x_j\biggl)\biggl((M_{d+1})_{i}\prod_{j\in [d]\backslash i}x_j\biggl)
 = (M_{d+1})_i\left(\frac{x_1\cdots x_d}{x_i} \right)^2,\end{align*}
where we use the notation $g_{j} = \frac{\partial g}{\partial x_j}$ 
and $g^{j} = g|_{x_j=0}$. 
It follows that $(M_{d+1})_i$ is a square in $\F$. 
Since $M_{d+1}$ has rank $\leq$ one, 
then for some choice of square root $v_i = \sqrt{(M_{d+1})_i}\in \F$, the matrix $M_{d+1}$ equals $vv^T$ with $v = (v_1, \hdots, v_d)^T$. 
\end{proof}

Using Corollary~\ref{cor:SquareDelInvariance}, this immediately gives the invariance of the set of determinantal polynomials.

\begin{cor}
The set of polynomials in $\F[{\bf x}]_{\leq {\bf 1}}$ with a determinantal representation \eqref{eq:det} is 
invariant under the action of ${\rm SL}_2(\F)^n \rtimes S_n$. 
\end{cor}

Together, Theorems~\ref{thm:square}~and~\ref{thm:GenDet} characterize the set of determinantal polynomials in $\F[{\bf x}]$. 

\begin{cor} \label{cor:MA_detrep}
A polynomial $f = \sum_{S\subseteq [n]}a_S{\bf x}^{[n]\backslash S} \in \F[{\bf x}]_{\leq {\bf 1}}$ 
has a determinantal representation \eqref{eq:det}
if and only if 
\begin{itemize}
\item[(i)] for every $i,j\in [n]$ and  $S\subseteq [n]\backslash \{i,j\}$, $a_{S\cup i}a_{S\cup j} - a_Sa_{S\cup ij}$ is a square over $\F$, and  
\item[(ii)] for every $\gamma\in \SL_2(\F)^n \rtimes S_n$, $(\gamma\cdot{\rm HypDet})({\bf a}) =0.$
\end{itemize}
\end{cor}

\begin{proof}
By Theorem~\ref{thm:GenDet}, $f$ has a determinantal representation \eqref{eq:det} if and only if for all $i,j$ 
$\Delta_{ij}f$ is a square in $\F[{\bf x}]$. 
Since $\Delta_{ij}f$ has degree $\leq 2$ in each variable, 
Theorem~\ref{thm:square}  implies that 
$\Delta_{ij}f =  \sum_{\alpha \in \{0,1,2\}^n}c_{\alpha}{\bf x}^{\alpha}$ 
is a square if and only if for every $\beta\in \{0,1\}^n$, $c_{2\beta}$ is a square in $R$ 
and $ {\bf c} = (c_{\alpha})_{\alpha \in \{0,1,2\}^n}$ satisfy 
the images of 
\begin{equation}\label{eq:squareC}
	(c_{(1,{\bf 0})})^2-4c_{(0,{\bf 0})} c_{(2,{\bf 0})}=0
\end{equation}
under the action $\SL_2(\F) \rtimes S_n$. This is in turn equivalent to the condition that for every $\gamma$ in $\SL_2(\F)^n \rtimes S_n$, $(\gamma\cdot{\rm HypDet})({\bf a}) =0$ and for every $i,j\in [n]$ and  $S\subseteq [n]\backslash \{i,j\}$, $a_{S\cup i}a_{S\cup j} - a_Sa_{S\cup ij} $ is a square in $\F$.

To see this, consider $S\subseteq [n]\backslash \{i,j\}$ and let $\beta \in \{0,1\}^n$ denote the indicator vector of $[n]\backslash (S\cup ij)$. 
We claim that that the coefficient of ${\bf x}^{2\beta}$ 
in $\Delta_{ij}(f)$ equals $a_{S\cup i}a_{S\cup j} - a_Sa_{S\cup ij} $. 
Since $f$, $\frac{\partial f}{\partial x_i}$, $\frac{\partial f}{\partial x_j}$,  and $\frac{\partial^2f}{\partial x_i \partial x_j}$ 
have degree $\leq 1$ in each variable, only the ${\bf x}^{\beta}$ terms in each of these polynomials 
contribute to the ${\bf x}^{2\beta}$ term in 
$\Delta_{ij}(f) = \frac{\partial f}{\partial x_i}\frac{\partial f}{\partial x_j} - f \frac{\partial^2f}{\partial x_i \partial x_j}$. 
That is,
\[
{\rm coeff}\left(\Delta_{ij}(f), {\bf x}^{2\beta}\right)= 
{\rm coeff}\left(\frac{\partial f}{\partial x_i}, {\bf x}^{\beta}\right)\cdot {\rm coeff}\left(\frac{\partial f}{\partial x_j}, {\bf x}^{\beta}\right)
- 
{\rm coeff}\left(f, {\bf x}^{\beta}\right)\cdot {\rm coeff}\left(\frac{\partial^2 f}{\partial x_i\partial x_j}, {\bf x}^{\beta}\right).
\]
Note that the coefficient of ${\bf x}^{\beta}$ in $f$ is $a_{S\cup ij}$. 
The coefficient of ${\bf x}^{\beta}$ in $\frac{\partial f}{\partial x_i}$ equals the coefficient of $x_i \cdot {\bf x}^{\beta}$ in $f$, 
which is $a_{S\cup j}$. Similarly, the coefficients of ${\bf x}^{\beta}$ in  $\frac{\partial f}{\partial x_j}$ and 
$\frac{\partial^2 f}{\partial x_i\partial x_j}$ are $a_{S\cup i}$ and $a_S$, respectively. 
\end{proof}

\subsection{Connections with the Grassmannian}
Given a $d\times n$ matrix $V$ of full-rank $d$, consider the 
polynomial $f$ from \eqref{eq:det} with $W = 0$: 
\begin{equation}\label{eq:CauchyBinet}
f({\bf x}) =\lambda  \det\left( V {\rm diag}(x_1, \hdots, x_n) V^T \right) = \lambda \det\left(\sum_{i=1}^n x_i v_iv_i^T\right)
= \lambda \sum_{S \in \binom{[n]}{d}} \bigl(V_S\bigl)^2 {\bf x}^S. 
\end{equation}
Here $\binom{[n]}{d}$ denotes the collection of size-$d$ subsets of $[n]$ and $V_S$ denotes the $d\times d$ minor of $V$ obtained by taking columns indexed by $S$. 
If $V$ has full rank $d$, the coefficients of $f$ are the squares of the Pl\"ucker coordinates 
given by the Pl\"ucker embedding of the rowspan of $V$ into $\Gr_{\F}(d,n)$. Otherwise $f$ is identically zero. 

Formally, consider the Pl\"ucker embedding of $\Gr_{\F}(d,n)$ into $\P^{\binom{n}{d}-1}(\F)$. 
Given a subspace $L\subseteq \F^n$ of dimension $d$, its image in $\P^{\binom{n}{d}-1}(\F)$
is the length-$\binom{n}{d}$ vector of $d\times d$ minors of any $d\times n$ matrix $V$ whose rowspan 
equals $L$. 
The map $[p_S]_S \mapsto [(p_S)^2]_S$ defines a morphism $\P^{\binom{n}{d}-1}(\F) \to \P^{\binom{n}{d}-1}(\F)$. Let  $\Gr_{\F}^2(d,n)$ denote the image of $\Gr_{\F}(d,n)$ under this morphism. 
Corollary~\ref{cor:MA_detrep} then gives an immediate characterization of $\Gr_{\F}^2(d,n)$
via the hyperdeterminantal equations in $\F^{2^n}$. In fact, 
setting $x_n=1$ in \eqref{eq:CauchyBinet}, 
we can study this image via 
multiaffine determinantal representations in the variables $x_1, \hdots, x_{n-1}$ 
and use the hyperdeterminantal equations in $\F^{2^{n-1}}$.

\begin{cor} 
Let ${\bf q} = (q_S)_{S\in \binom{[n]}{d}}\in \P^{\binom{n}{d}-1}(\F)$ and let 
${\bf a}\in \P^{2^{n-1}-1}(\F)$ denote the vector 
given by 
\[
a_S = \begin{cases} 
q_S & \text{ if } S\subseteq[n-1], |S|=d\\
q_{S\cup n} &  \text{ if } S\subseteq[n-1], |S|=d-1\\
0 & \text{ otherwise }. 
\end{cases}
\]
The vector ${\bf q}$ belongs to $\Gr_{\F}^2(d,n)$ if and only if  
\begin{itemize}
\item[(i)] for all $i,j\in [n]$ and $S\subseteq [n]\backslash \{i,j\}$, $|S|=d-1$, 
$q_{S\cup i}q_{S\cup j}$ is a square over $\F$, and  
\item[(ii)] for every $\gamma\in \SL_2(\F)^{n-1} \rtimes S_{n-1}$, $(\gamma\cdot{\rm HypDet})({\bf a}) =0.$
\end{itemize}
\end{cor}
\begin{proof} ($\Rightarrow$) This follows from applying Corollary~\ref{cor:MA_detrep} to \eqref{eq:CauchyBinet}.

($\Leftarrow$) 
Let $f = \sum_{S\in \binom{[n]}{d}} q_S {\bf x}^S$. 
This equals $f = f_{\bf b}$ where ${\bf b}\in \F^{2^n}$ 
is given by $b_S = q_S$ for $|S|=d$ and $b_S = 0$ otherwise. 
If ${\bf a} = (a_S)_S$ 
is defined as above, then $f_{\bf a}$ equals the restriction of $f$ to $x_n=1$
and $f$ is the homogenezation of $f_{\bf a}$ to degree $d$ with homogenizing variable $x_n$.

For any $i\neq j\in [n-1]$ and $S\subseteq [n-1]\backslash\{i,j\}$, 
the sizes of $S$ and $S\cup\{ i, j\}$ differ by two, implying that $a_{S}a_{S\cup ij}$ equals zero. 
Then $a_{S\cup i}a_{S\cup j} - a_{S}a_{S\cup ij} = a_{S\cup i}a_{S\cup j} $ is a square in $\F$ by assumption (i). 
By Theorem~\ref{thm:GenDet} and Corollary~\ref{cor:MA_detrep}, 
$\Delta_{ij}(f_{\bf a})$ is a square in $\F[x_1, \hdots, x_{n-1}]$ for all $i,j \in [n-1]$, implying that 
$\Delta_{ij}(f)$ is a square in $\F[{\bf x}]$ for all $i,j \in [n-1]$. In particular, 
${\rm Discr}_{x_n}(\Delta_{ij}(f))$ is identically zero. 
One can check that for any $i,j,n$, ${\rm Discr}_{x_n}(\Delta_{ij}(f))$ equals 
${\rm Discr}_{x_j}(\Delta_{in}(f))$. This shows that 
$(\gamma\cdot{\rm HypDet})({\bf b}) =0$ for all $\gamma\in \SL_2(\F)^{n} \rtimes S_{n}$.
Assumption (i) implies that $b_{S\cup i}b_{S\cup j} - b_Sb_{S\cup ij}$ is a square in $\F$ 
for all $i,j\in [n]$ and $S\subseteq[n]\backslash\{i,j\}$, since this is either zero or of the form 
$q_{S\cup i}q_{S\cup j}$.
Corollary~\ref{cor:MA_detrep} then gives a representation 
$f = \lambda\det(\sum_{i=1}^n x_iv_iv_i^T + W)$ where $v_i\in \F^d$ and $W\in {\rm Sym}_d(\F)$.  
The polynomial $\lambda\det(\sum_{i=1}^n x_iv_iv_i^T  + yW)\in \F[\x, y]$
equals the homogenization of $f$ to degree $d$. Since $f$ is already homogeneous of degree $d$, 
this equals $f$ and belongs to $\F[\x]$. 
Specializing to $y=0$ gives the desired representation $f = \lambda\det(\sum_{i=1}^n x_iv_iv_i^T)$.
\end{proof}

\begin{example}($d=2$, $n=4$)
The Grassmannian $\Gr_\F(2,4)$ is cut out by one Pl\"ucker relation 
$p_{12}p_{34} - p_{13}p_{24}  + p_{14}p_{23}=0$ in $\P^5(\F)$. 
Taking $q_{ij} = p_{ij}^2$ and eliminating the variables $p_{ij}$ gives the 
defining equation 
\[
q_{12}^2q_{34}^2 + q_{13}^2q_{24}^2 +q_{14}^2q_{23}^2
-2q_{12}q_{13}q_{24}q_{34} - 2 q_{12}q_{14}q_{23}q_{34}
- 2 q_{13}q_{14}q_{23}q_{24}=0
\]
for  $\Gr^2_\F(2,4)$. This is exactly the hyperdeterminant ${\rm HypDet}({\bf a})$ 
where ${\bf a} = (a_{S})_{S\subseteq [3]} \in \F^{2^3}$ is given by 
 $a_{\emptyset} = a_{123}=0$, $a_{i} = q_{i4}$ and $a_{ij} = q_{ij}$ 
for all $i,j\in [3]$. 
\end{example}

\subsection{Determinantal representations in higher degrees}
For any $r\in \Z_+$, let ${\rm Sym}_n^r(\F)$ denote the set of symmetric matrices 
over $\F$ that can be written as a sum of $r$ rank-one matrices over $\F$, i.e.  
\[
{\rm Sym}_n^r(\F) = \left\{ \sum_{i=1}^r v_iv_i^T :  \ v_1, \hdots, v_r \in \F^n\right\}. 
\]
If $\F$ is algebraically closed with ${\rm char}(\F)\neq 2$, this is just the set of matrices 
of rank $\leq r$.  For $\F=\R$ this is the set of positive semidefinite 
matrices of rank $\leq r$.

\begin{thm}\label{thm:DetInvariance}
The set of polynomials $\F[{\bf x}]_{\leq {\bf d}}$ 
with a determinantal representation 
\begin{equation}\label{eq:higherDegDetRep}
f = \lambda \det\left(\sum_{i=1}^n x_iA_i + B\right) 
\text{ with $A_i \in {\rm Sym}_m^{d_i}(\F)$  and $B \in  {\rm Sym}_m(\F)$} 
\end{equation}
for some $m\in \N$ is invariant under the action of ${\rm SL}_2(\F)^n \rtimes S_n$. 
\end{thm}

\begin{proof}
The invariance under the action of $S_n$ is immediate. It remains to 
check the invariance under  ${\rm SL}_2(\F)^n$.
Suppose that $f = \det(\sum_{i=1}^n x_iA_i + B)$, where
$A_i \in {\rm Sym}_m^{d_i}(\F)$ and $B \in  {\rm Sym}_m(\F)$.
First, suppose that ${\bf d}={\bf 1}$ and let $\gamma\in {\rm SL}_2(\F)^n$. By Corollary~\ref{cor:SquareDelInvariance},
$\Delta_{ij}(\gamma\cdot f)$ is a square for all $i,j$. 
Then by Theorem~\ref{thm:GenDet}, $\gamma\cdot f$ has a determinantal representation as in \eqref{eq:det}.

Now consider arbitrary ${\bf d} = (d_1, \hdots, d_n)$. 
By definition, we can write each matrix $A_i$ as a sum of 
$d_i$ matrices $A_{ij}$ each of the form $vv^T$ for some $v\in \F^m$. 
Then consider
\[
F = \det\left(\sum_{i=1}^n \sum_{j=1}^{d_i} y_{ij}A_{ij} + B\right)
\ \in \ \F\bigl[y_{ij} : i\in [n], j\in [d_i]\bigl].
\]
Note that there is an inclusion $\phi:{\rm SL}_2(\F)^n\to {\rm SL}_2(\F)^{d_1+\hdots+ d_n}$, given by $(\phi(\gamma))_{ij} = \gamma_i$ for all $i\in [n]$ and $j\in [d_i]$. By construction, the restriction of $\phi(\gamma) \cdot F$ 
given by $y_{ij} = x_i$ for all $i,j$ gives $\gamma\cdot f$. That is, 
\[
\left(\phi(\gamma) \cdot F\right)|_{y_{ij} = x_{i} } = \gamma\cdot f. 
\]
By the case ${\bf d}={\bf 1}$, $\phi(\gamma) \cdot F$ is determinantal.  That is, there are some matrices $C_{11}, \hdots, C_{nd_n}, D$ with $C_{ij}\in {\rm Sym}_m^1(\F)$  so that 
$\phi(\gamma) \cdot F$ equals $\det( \sum_{i=1}^n \sum_{j=1}^{d_i} y_{ij}C_{ij} + D)$. 
 Then $\gamma\cdot f $ equals $ \det( \sum_{i=1}^n x_{i}C_{i} + D)$
where $C_i = \sum_{j=1}^{d_i}C_{ij}\in {\rm Sym}_m^{d_i}(\F)$. 
\end{proof}

One motivation for studying such polynomials comes from definite determinantal representations over $\R$ 
and their connection with stable polynomials. 
A real polynomial $f\in \R[{\bf x}]$ is \emph{stable} if it has no zeros with strictly positive imaginary parts, i.e. 
$f({\bf z}) \neq 0$ for all ${\bf z}\in \C^n$ with ${\rm Im}({\bf z})\in \R_+^n$. 
Equivalently, $f$ is stable if and only if the polynomial $f(t{\bf v} + {\bf w})\in \R[t]$ is real-rooted for 
all ${\bf v}\in \R_+^n$ and ${\bf w}\in \R^n$.  
Over $\F = \R$, any polynomial with a determinantal representation of the form \eqref{eq:higherDegDetRep} is stable (see, e.g.~\cite[Prop.~2.1]{WagnerSurvey}), but not every stable 
polynomial has such a representation (see \cite{BrandenDetRep}). 

The action of ${\rm SL}_2(\R)$ on $\C$ given by $\tiny{\begin{pmatrix} a & b \\ c & d \end{pmatrix}} \cdot z = \frac{az+b}{cz+d}$  preserves the upper half plane $\{z \in \C: {\rm Im}(z)>0\}$ and so the set of stable polynomials 
in $\R[{\bf x}]$ is invariant under the action of ${\rm SL}_2(\R)^n\rtimes S_n$. 
One consequence of Theorem~\ref{thm:DetInvariance} is that the set of polynomials with a semidefinite determinantal representation 
is also invariant under the action of this group. 

\begin{cor}
The set of polynomials in $\R[{\bf x}]_{\leq {\bf d}}$ with a determinantal representation 
\[
f = \lambda \det\left(\sum_{i=1}^n x_iA_i + B\right) 
\text{ with $A_1, \hdots, A_n, B \in {\rm Sym}_m(\R)$  and $A_1, \hdots, A_n\succeq 0$} 
\]
for some $m\in \N$ and $\lambda \in \R$ is invariant under the action of ${\rm SL}_2(\R)^n\rtimes S_n$.
\end{cor}

%
%

%

\begin{thebibliography}{BDOvdD12}

\bibitem[BBL09]{BBL09}
Julius Borcea, Petter Br\"{a}nd\'{e}n, and Thomas~M. Liggett.
\newblock Negative dependence and the geometry of polynomials.
\newblock {\em J. Amer. Math. Soc.}, 22(2):521--567, 2009.

\bibitem[BC15]{BoussairiChergui15}
Abderrahim Boussa\"{\i}ri and Brahim Chergui.
\newblock Skew-symmetric matrices and their principal minors.
\newblock {\em Linear Algebra Appl.}, 485:47--57, 2015.

\bibitem[BDOvdD12]{BrauldiKeaettOleskyDriessche12}
R.~A. Brualdi, L.~Deaett, D.~D. Olesky, and P.~van~den Driessche.
\newblock The principal rank characteristic sequence of a real symmetric
  matrix.
\newblock {\em Linear Algebra Appl.}, 436(7):2137--2155, 2012.

\bibitem[Br{\"a}07]{Branden07}
Petter Br{\"a}nd\'{e}n.
\newblock Polynomials with the half-plane property and matroid theory.
\newblock {\em Adv. Math.}, 216(1):302--320, 2007.

\bibitem[Br{\"a}11]{BrandenDetRep}
Petter Br{\"a}nd{\'e}n.
\newblock Obstructions to determinantal representability.
\newblock {\em Adv. Math.}, 226(2):1202--1212, 2011.

\bibitem[Dod67]{Dodgson}
Charles~Lutwidge Dodgson.
\newblock Condensation of determinants, being a new and brief method for
  computing their arithmetical values.
\newblock {\em Proc. R. Soc. Lond.}, 15:150--155, 1867.

\bibitem[GKVW13]{GrinshpanKaliuzhnyiWoerdeman13}
Anatolii Grinshpan, Dmitry~S. Kaliuzhnyi-Verbovetskyi, and Hugo~J. Woerdeman.
\newblock Norm-constrained determinantal representations of multivariable
  polynomials.
\newblock {\em Complex Anal. Oper. Theory}, 7(3):635--654, 2013.

\bibitem[GT06a]{GriffinTsatsomeros06}
Kent Griffin and Michael~J. Tsatsomeros.
\newblock Principal minors. {I}. {A} method for computing all the principal
  minors of a matrix.
\newblock {\em Linear Algebra Appl.}, 419(1):107--124, 2006.

\bibitem[GT06b]{GriffinTsatsomeros06s}
Kent Griffin and Michael~J. Tsatsomeros.
\newblock Principal minors. {II}. {T}he principal minor assignment problem.
\newblock {\em Linear Algebra Appl.}, 419(1):125--171, 2006.

\bibitem[HO17]{HuangOeding17}
Huajun Huang and Luke Oeding.
\newblock Symmetrization of principal minors and cycle-sums.
\newblock {\em Linear Multilinear Algebra}, 65(6):1194--1219, 2017.

\bibitem[HS07]{HoltzSturmfels07}
Olga Holtz and Bernd Sturmfels.
\newblock Hyperdeterminantal relations among symmetric principal minors.
\newblock {\em J. Algebra}, 316(2):634--648, 2007.

\bibitem[KLS08]{KodiyalamSwan08}
Vijay Kodiyalam, T.~Y. Lam, and R.~G. Swan.
\newblock Determinantal ideals, {P}faffian ideals, and the principal minor
  theorem.
\newblock In {\em Noncommutative rings, group rings, diagram algebras and their
  applications}, volume 456 of {\em Contemp. Math.}, pages 35--60. Amer. Math.
  Soc., Providence, RI, 2008.

\bibitem[KP14]{KenyonPemantle14}
Richard Kenyon and Robin Pemantle.
\newblock Principal minors and rhombus tilings.
\newblock {\em J. Phys. A}, 47(47):474010, 17, 2014.

\bibitem[KPV15]{KPV15}
Mario Kummer, Daniel Plaumann, and Cynthia Vinzant.
\newblock Hyperbolic polynomials, interlacers, and sums of squares.
\newblock {\em Math. Program.}, 153(1, Ser. B):223--245, 2015.

\bibitem[KT12]{DPP&MachineLearning}
Alex Kulesza and Ben Taskar.
\newblock {\em Determinantal Point Processes for Machine Learning}.
\newblock Now Publishers Inc., Hanover, MA, USA, 2012.

\bibitem[LS09]{LinSturmfels09}
Shaowei Lin and Bernd Sturmfels.
\newblock Polynomial relations among principal minors of a {$4\times
  4$}-matrix.
\newblock {\em J. Algebra}, 322(11):4121--4131, 2009.

\bibitem[Oed09]{OedingThesis}
Luke Oeding.
\newblock {\em G-Varieties and the Principal Minors of Symmetric Matrices}.
\newblock PhD thesis, Texas A\&M University, 2009.

\bibitem[Oed11]{Oeding11}
Luke Oeding.
\newblock Set-theoretic defining equations of the variety of principal minors
  of symmetric matrices.
\newblock {\em Algebra Number Theory}, 5(1):75--109, 2011.

\bibitem[PV13]{PV13}
Daniel Plaumann and Cynthia Vinzant.
\newblock Determinantal representations of hyperbolic plane curves: an
  elementary approach.
\newblock {\em J. Symbolic Comput.}, 57:48--60, 2013.

\bibitem[RKT15]{RisingKuleszaTaskar15}
Justin Rising, Alex Kulesza, and Ben Taskar.
\newblock An efficient algorithm for the symmetric principal minor assignment
  problem.
\newblock {\em Linear Algebra Appl.}, 473:126--144, 2015.

\bibitem[vGM19]{char2}
Bert van Geemen and Alessio Marrani.
\newblock Lagrangian {G}rassmannians and spinor varieties in characteristic
  two.
\newblock {\em SIGMA Symmetry Integrability Geom. Methods Appl.}, 15:Paper No.
  064, 22, 2019.

\bibitem[Wag11]{WagnerSurvey}
David~G. Wagner.
\newblock Multivariate stable polynomials: theory and applications.
\newblock {\em Bull. Amer. Math. Soc. (N.S.)}, 48(1):53--84, 2011.

\bibitem[Whe15]{Wheeler}
Ashley~K. Wheeler.
\newblock Ideals generated by principal minors.
\newblock {\em Illinois J. Math.}, 59(3):675--689, 2015.

\end{thebibliography}
\bibliographystyle{alpha}

\end{document}